\pdfoutput=1
\RequirePackage{ifpdf}
\ifpdf 
\documentclass[pdftex]{sigma}
\else
\documentclass{sigma}
\fi

\numberwithin{equation}{section}

\newtheorem{Theorem}{Theorem}[section]
\newtheorem*{Theorem*}{Theorem}
\newtheorem{Corollary}[Theorem]{Corollary}
\newtheorem{Lemma}[Theorem]{Lemma}

\theoremstyle{definition}
\newtheorem{Definition}[Theorem]{Definition}

\newtheorem{Remark}[Theorem]{Remark}

\DeclareMathOperator{\id}{id}
\DeclareMathOperator{\tr}{tr}
\DeclareMathOperator{\ad}{ad}

\newcommand{\ve}{\varepsilon}

\allowdisplaybreaks

\begin{document}

\allowdisplaybreaks

\newcommand{\arXivNumber}{2512.06340}

\renewcommand{\PaperNumber}{089}

\FirstPageHeading

\ShortArticleName{Subalgebras of the Universal Central Extension of $\mathfrak{sl}(n)[u^{\pm},v]$ from Two Involutions}

\ArticleName{Subalgebras of the Universal Central Extension\\ of $\boldsymbol{\mathfrak{sl}(n)[u^{\pm},v]}$ from Two Involutions}

\Author{Mamoru UEDA}

\AuthorNameForHeading{M.~Ueda}

\Address{Graduate School of Mathematical Sciences, The University of Tokyo, 3-8-1 Komaba, Meguro, Tokyo, 153-8914, Japan}
\Email{\mail{mueda@ms.u-tokyo.ac.jp}}

\ArticleDates{Received January 18, 2026, in final form August 31, 2026; Published online September 18, 2026}

\Abstract{The affine Yangian associated with \smash{$\widehat{\mathfrak{sl}}(n)$} is a deformation of the universal enveloping algebra of the universal central extension of \smash{$\mathfrak{sl}(n)[u^{\pm},v]$}.
In this article, we consider two involutions of the universal central extension of \smash{$\mathfrak{sl}(n)[u^{\pm},v]$} and give a presentation of the fixed subalgebras of these involutions. By deforming one fixed subalgebra, we can define the 1-parameter twisted affine Yangian \smash{$TY_{\hbar}(\widehat{\mathfrak{so}}(n))$}, which is a coideal of the affine Yangian associated with \smash{$\widehat{\mathfrak{sl}}(n)$}. Another fixed subalgebra is expected to define the 2-parameter twisted affine Yangian by deforming the presentation given in this article.}


\Classification{17B05; 17B35; 17B69}

\section{Introduction}

Drinfeld~\cite{D1,D2} introduced the finite Yangian $Y_h(\mathfrak{g})$ associated with a finite-dimensional simple Lie algebra $\mathfrak{g}$. The finite Yangian $Y_h(\mathfrak{g})$ is a quantum group which is a deformation of a current algebra $\mathfrak{g}\otimes\mathbb{C}[z]$. The finite Yangian has several presentations: the Drinfeld presentation~\cite{D2}, finite presentation~\cite{GNW} and so on.

Olshanskii \cite{O} introduced the twisted Yangian associated with an orthogonal Lie algebra~$\mathfrak{so}(n)$ or a symplectic Lie algebra $\mathfrak{sp}(n)$. These twisted Yangians are coideals of the finite Yangian of type $A$. Similarly to the finite Yangian, the twisted Yangian also has several presentations. Olshanskii defined the twisted Yangian by the RTT presentation. Belliard and Regelskis~\cite{BR1} constructed the Drinfeld $J$ presentation of the twisted Yangian. Lu--Wang--Zhang~\mbox{\cite{LWZ0,LWZ}} and Lu--Zhang~\cite{LWZ1} gave the Drinfeld presentation of the twisted Yangian. As for the minimalistic presentation of the twisted Yangian, Harako and the author~\cite{HU} constructed it for type $D$ setting and Lu~\cite{L} gave it for any setting.

By using the Drinfeld presentation of the finite Yangian, the definition of the Yangian can be naturally extended to a symmetrizable Kac--Moody Lie algebra. In particular, the Yangian associated with an affine Lie algebra is called the {\it affine Yangian}. For an affine Lie algebra, Guay--Nakajima--Wendlandt~\cite{GNW} constructed a finite presentation of the Yangian and showed that the Yangian associated with an affine Lie algebra has a coproduct. Unlike finite Yangians, the affine Yangian is not a deformation of the current algebra. For instance, the affine Yangian associated with \smash{$\widehat{\mathfrak{sl}}(n)$} is a deformation of the universal central extension of \smash{$\mathfrak{sl}(n)\bigl[u^{\pm1},v\bigr]$}. The~universal central extension of $\mathfrak{sl}(n)\bigl[u^{\pm1},v\bigr]$ is isomorphic to
\begin{equation*}
\mathfrak{sl}(n)\bigl[u^{\pm1},v\bigr]\oplus\Omega\bigl(\mathbb{C}\bigl[u^{\pm1},v\bigr]\bigr)/ d \mathbb{C}\bigl[u^{\pm1},v\bigr],
\end{equation*}
where \raisebox{0.2pt}{$\Omega\bigl(\mathbb{C}\bigl[u^{\pm1},v\bigr]\bigr)$} is the module of differentials of $\mathbb{C}\bigl[u^{\pm1},v\bigr]$. This Lie algebra is spanned by~\raisebox{-0.2pt}{$\mathfrak{sl}(n)\bigl[u^{\pm1},v\bigr]$} and $\bigl\{v^{r}u^{-1}{d}u\mid r\in\mathbb{Z}_{\geq0}\bigr\}$.

The Lie algebra $\mathfrak{sl}(n)\bigl[u^{\pm1},v\bigr]$ has the following involution:
\begin{align*}
\omega\colon\ \mathfrak{sl}(n)\bigl[u^{\pm1},v\bigr]\to\mathfrak{sl}(n)\bigl[u^{\pm1},v\bigr],\qquad Xu^rv^s\mapsto -(-1)^sX^{\mathsf{T}}u^{-r}v^s,
\end{align*}
where $X^{\mathsf{T}}$ is the transpose of $X$ as a matrix. Set $\mathfrak{BD}(n)$ as a subalgebra of the universal central extension of $\mathfrak{sl}(n)\bigl[u^{\pm1},v\bigr]$ spanned by $\mathfrak{sl}(n)\bigl[u^{\pm1},v\bigr]^\omega$ and $\bigl\{v^{2r+1}u^{-1}{d}u\mid r\geq0\bigr\}$. In Sections~\ref{sec:3} and~\ref{sec:4}, we give a presentation of $\mathfrak{BD}(n)$.
\begin{Theorem}
The Lie algebra $\mathfrak{BD}(n)$ has a presentation whose number of generators is finite.
\end{Theorem}
In Section~\ref{sec:4}, we define the 1-parameter twisted affine Yangian associated with $\mathfrak{so}(n)$ as a~deformation of \smash{$U\bigl(\mathfrak{sl}(n)\bigl[u^{\pm1},v\bigr]^\omega\bigr)$} and show that it becomes a coideal of the affine Yangian of type~$A$ by the same way as~\cite{L}.
Guay~\cite{Gu1} defined the affine Yangian $Y_{\hbar,\ve}\bigl(\widehat{\mathfrak{sl}}(n)\bigr)$ with two parameters~$\hbar$,~$\ve$ associated with \smash{$\widehat{\mathfrak{sl}}(n)$} by extending the finite presentation of the Yangian associated with~$\mathfrak{sl}(n)$. The definition of the 1-parameter twisted affine Yangian is well defined and has no problem. However, it could be challenging to extend it to the 2-parameter.

Recently, the Yangian has been actively used for the study of $W$-algebras. In the finite setting, there exists a relationship between the finite Yangian of type $A$ and a finite $W$-algebra of type~$A$ given by Brundan--Kleshchev~\cite{BK}. As for other types, a finite $W$-algebra is connected to the twisted Yangian (see~\cite{Bro,DKV,LPTTW}). In~\cite{U77}, the author defined the twisted affine Yangian of type~$C$ as a subalgebra of the 2-parameter affine Yangian $Y_{\hbar,\ve}\bigl(\widehat{\mathfrak{sl}}(n)\bigr)$ and constructed a relationship between the twisted affine Yangian of type $C$ and a rectangular $W$-algebra of type~$D$ in certain cases. However, this definition was not described by generators and defining relations. Thus, we would like to describe a definition of a two-parameter twisted affine Yangian by generators and defining relations.

We consider another involution of $\mathfrak{sl}(2n)\bigl[u^{\pm1},v\bigr]$ similarly to~\cite{HU}.
Set $I_n=\{\pm1,\pm2,\dots,\pm n\}$ and take a basis of $\mathfrak{gl}(2n)$ as $\{E_{i,j}\}_{i,j\in I_n}$. The Lie algebra $\mathfrak{gl}(2n)\bigl[u^{\pm1},v\bigr]$ has the following involution:
\begin{align*}
\tau\colon\ \mathfrak{gl}(2n)\bigl[u^{\pm1},v\bigr]\to\mathfrak{gl}(2n)\bigl[u^{\pm1},v\bigr],\qquad E_{i,j}u^rv^s\mapsto -(-1)^sE_{-j,-i}u^{-r}v^s.
\end{align*}
Set $D(n)$ as a subalgebra of the universal central extension of $\mathfrak{sl}(2n)\bigl[u^{\pm1},v\bigr]$ spanned by the fixed subalgebra $\mathfrak{sl}(2n)\bigl[u^{\pm1},v\bigr]^\tau$ and $\bigl\{v^{2r}u^{-1}{d}u\mid r\geq0\bigr\}$.
In Sections~\ref{sec:5} and~\ref{sec:6}, we construct a~presentation of~$D(2n)$.
\begin{Theorem}
The Lie algebra $D(n)$ has a presentation whose number of generators is finite.
\end{Theorem}
We expect that there exists a 2-parameter twisted affine Yangian deforming this presentation. Actually, in~\cite{U78}, the author suggested the 2-parameter twisted affine Yangian $TY_{\hbar,\ve}(\widehat{\mathfrak{so}}(2n))$ and constructed a homomorphism from $TY_{\hbar,\ve}(\widehat{\mathfrak{so}}(8))$ to the rectangular $W$-algebra associated with~$\mathfrak{sp}(16)$ and a nilpotent element of type $\bigl(2^8\bigr)$. However, this definition is very complicated and we expect that the defining relations of $TY_{\hbar,\ve}(\widehat{\mathfrak{so}}(2n))$ can be further simplified.

\section[Universal central extension of sl(n){[u\^{}\{pm1\},v]}]{Universal central extension of $\boldsymbol{\mathfrak{sl}(n)\bigl[u^{\pm1},v\bigr]}$}
In this section, we recall the work of~\cite{MRY}, which gives a presentation of the universal central extension of $\mathfrak{sl}(n)\bigl[u^{\pm1},v\bigr]$. For a perfect Lie algebra $\mathfrak{g}$, a central extension of $\mathfrak{g}$ is a pair consisting of a~Lie algebra \smash{$\widetilde{\mathfrak{g}}$} and a surjective homomorphism $\pi\colon\widetilde{\mathfrak{g}}\to\mathfrak{g}$ whose kernel is contained in the center of~\smash{$\widetilde{\mathfrak{g}}$}. The central extension $(\widetilde{\mathfrak{g}},\pi)$ is called the universal central extension of $\mathfrak{g}$ if for every central extension $(\mathfrak{e},\psi)$ there exists a unique homomorphism $\phi\colon\widetilde{\mathfrak{g}}\to\mathfrak{e}$ satisfying that $\psi\circ\phi=\pi$.

In~\cite{Kassel}, Kassel gave an explicit presentation of the universal central extension of $\mathfrak{sl}(n)\bigl[u^{\pm1},v\bigr]$.
We set $\bigl(\Omega\bigl(\mathbb{C}\bigl[u^{\pm1},v\bigr]\bigr),d\bigr)$ as the module of differentials of $\mathbb{C}\bigl[u^{\pm1},v\bigr]$ in the same way as in~\cite[Section~2]{MRY}. We take $\{a_i\}$ as a basis of $\mathbb{C}\bigl[u^{\pm1},v\bigr]$. Let $F$ be a free module of $\mathbb{C}\bigl[u^{\pm1},v\bigr]$ whose basis is \smash{$\bigl\{\widetilde{d} a_i\bigr\}$}. We regard $F$ as a 2-sided \smash{$\mathbb{C}\bigl[u^{\pm1},v\bigr]$} module by setting \smash{$a\bigl(\widetilde{d} b\bigr)=\bigl(\widetilde{d} b\bigr)a$} for~any~\smash{$a,b\in\mathbb{C}\bigl[u^{\pm1},v\bigr]$}. We set a linear map \smash{$\widetilde{d} \colon \mathbb{C}\bigl[u^{\pm1},v\bigr]\to F$} as \smash{$\widetilde{d} \bigl(\sum_{i}c_ia_i\bigr)=\sum_{i}c_i\widetilde{d} a_i$} and $K$ as a~submodule of~$F$ generated by \smash{$\widetilde{d}(ab)=(\widetilde{d}a)b+a\bigl(\widetilde{d} b\bigr)$} for $a,b\in\mathbb{C}\bigl[u^{\pm1},v\bigr]$. We denote $F/K$ by~\smash{$\Omega\bigl(\mathbb{C}\bigl[u^{\pm1},v\bigr]\bigr)$}. We also set \smash{$d\colon\mathbb{C}\bigl[u^{\pm1},v\bigr]\to\Omega\bigl(\mathbb{C}\bigl[u^{\pm1},v\bigr]\bigr)$} as a homomorphism induced by \smash{$\widetilde{d}$}.

Let \smash{$\mathfrak{sl}(n)\otimes\mathbb{C}\bigl[u^{\pm1},v\bigr]\oplus\Omega\bigl(\mathbb{C}\bigl[u^{\pm1},v\bigr]\bigr)/d\mathbb{C}\bigl[u^{\pm1},v\bigr]$} be the Lie algebra whose commutator relations are given by
\begin{gather*}
[z_1\otimes a_1,z_2\otimes a_2]=[z_1,z_2]\otimes a_1\cdot a_2+\kappa(z_1,z_2)a_2da_1,\\
\bigl[\mathfrak{sl}(n)\otimes\mathbb{C}\bigl[u^{\pm1},v\bigr],\Omega^1\bigl(\mathbb{C}\bigl[u^{\pm1},v\bigr]\bigr)/d\mathbb{C}\bigl[u^{\pm1},v\bigr]\bigr]=0,\\
\bigl[\Omega^1\bigl(\mathbb{C}\bigl[u^{\pm1},v\bigr]\bigr)/d\mathbb{C}\bigl[u^{\pm1},v\bigr],\Omega^1\bigl(\mathbb{C}\bigl[u^{\pm1},v\bigr]\bigr)/d\mathbb{C}\bigl[u^{\pm1},v\bigr]\bigr]=0.
\end{gather*}
We can naturally define a surjective homomorphism
\[
\pi\colon\ \mathfrak{sl}(n)\otimes\mathbb{C}\bigl[u^{\pm1},v\bigr]\oplus\Omega\bigl(\mathbb{C}\bigl[u^{\pm1},v\bigr]\bigr) / d\mathbb{C}\bigl[u^{\pm1},v\bigr]\to\mathfrak{sl}(n)\otimes\mathbb{C}\bigl[u^{\pm1},v\bigr].
\]
Since the kernel of $\pi$ is $\Omega\bigl(\mathbb{C}\bigl[u^{\pm1},v\bigr]\bigr)/d\mathbb{C}\bigl[u^{\pm1},v\bigr]$, $\pi$ gives a central extension.
Moreover, by~\cite{Kassel}, $\mathfrak{sl}(n)\otimes\mathbb{C}\bigl[u^{\pm1},v\bigr]\oplus\Omega \bigl(\mathbb{C}\bigl[u^{\pm1},v\bigr]\bigr)/d\mathbb{C}\bigl[u^{\pm1},v\bigr]$ together with $\pi$ form the universal central extension of $\mathfrak{sl}(n)\otimes\mathbb{C}\bigl[u^{\pm1},v\bigr]$.

Moody--Rao--Yokonuma~\cite{MRY} gave a new presentation of this universal central extension, which we describe below.
Let us take a matrix \smash{$\bigl(a^A_{i,j}\bigr)_{0\leq i,j\leq n-1}$} as
\begin{equation*}
a^A_{i,j}=\begin{cases}
\hphantom{-}2&\text{if }i=j,\\
-1&\text{if }i=j\pm1,\\
-1&\text{if }(i,j)=(0,n-1),(n-1,0),\\
\hphantom{-}0&\text{otherwise}.
\end{cases}
\end{equation*}
Then, \smash{$\bigl(a^A_{i,j}\bigr)_{0\leq i,j\leq n-1}$} is the Cartan matrix of an affine Lie algebra $\smash{\widehat{\mathfrak{sl}}(n)}=\mathfrak{sl}(n)\otimes\mathbb{C}\bigl[t^{\pm1}\bigr]\oplus\mathbb{C}c$ whose commutator relations are given by
\begin{gather*}
[x\otimes t^u,y\otimes t^s]=[x,y]\otimes t^{u+s}+u\delta_{u+s,0}\tr(xy)c \qquad\text{for }x,y\in\mathfrak{sl}(n),\\
\bigl[c,\widehat{\mathfrak{sl}}(n)\bigr]=0,
\end{gather*}
where $\tr$ is the trace of $\mathfrak{gl}(n)$.
\begin{Definition}
We define $\mathfrak{A}(n)$ as the Lie algebra over $\mathbb{C}$ generated by
\begin{equation*}
\bigl\{x_{i,r}^\pm, h_{i,r}\mid 0\leq i\leq n-1,\,r\in\mathbb{Z}_{\geq0}\bigr\}
\end{equation*}
subject to the following defining relations:
\begin{gather}
[h_{i,r}, h_{j,s}] = 0,\label{Eq2.1}\\
\bigl[x_{i,r}^{+}, x_{j,s}^{-}\bigr] = \delta_{i,j} h_{i, r+s},\label{Eq2.2}\\
\bigl[h_{i,0}, x_{j,r}^{\pm}\bigr] = \pm a^A_{i,j} x_{j,r}^{\pm},\label{Eq2.4}\\
\bigl[h_{i,r+1}, x_{j,s}^{\pm}\bigr] = \bigl[h_{i,r}, x_{j,s+1}^{\pm}\bigr],\label{Eq2.5}\\
\bigl[x_{i, r+1}^{\pm}, x_{j, s}^{\pm}\bigr] = \bigl[x_{i, r}^{\pm}, x_{j, s+1}^{\pm}\bigr],\label{Eq2.8}\\
\sum_{\sigma\in S_{1-a_{i,j}}}\limits\bigl(\ad x_{i,r_{\sigma(1)}}^{\pm}\bigr)\cdots\bigl(\ad x_{i,r_{\sigma(1-a_{i,j})}}^{\pm}\bigr)\bigl(x_{j,s}^{\pm}\bigr)= 0 \qquad\text{if }i \neq j,\label{Eq2.10}
\end{gather}
where $S_l$ is a symmetric group of degree $l$. We also denote the universal enveloping algebra of~$\mathfrak{A}(n)$~by~$A(n)$.
\end{Definition}
By using \eqref{Eq2.8}, we obtain
\begin{equation*}
\bigl[x^\pm_{i,r},x^\pm_{i,s}\bigr]=\frac{1}{2}\bigl(\bigl[x^\pm_{i,r+s},x^\pm_{i,0}\bigr]+\bigl[x^\pm_{i,0},x^\pm_{i,r+s}\bigr]\bigr)=0.
\end{equation*}

Moody--Rao--Yokonuma showed that $\mathfrak{A}(n)$ is isomorphic to the universal central extension of~$\mathfrak{sl}(n)\otimes\mathbb{C}\bigl[u^{\pm1},v\bigr]$.
\begin{Theorem}[{\cite[Proposition 3.5]{MRY}}]
There exists an isomorphism
\begin{equation*}
\iota\colon\ \mathfrak{A}(n)\to \mathfrak{sl}(n)\otimes\mathbb{C}\bigl[u^{\pm1},v\bigr]\oplus\Omega\bigl(\mathbb{C}\bigl[u^{\pm1},v\bigr]\bigr)/d\mathbb{C}\bigl[u^{\pm1},v\bigr]
\end{equation*}
given by
\begin{gather*}
x^+_{i,r}=\begin{cases}
E_{i,i+1}v^r&\text{if }i\neq0,\\
E_{n,1}uv^r&\text{if }i=0,
\end{cases}\qquad
x^-_{i,r}=\begin{cases}
E_{i+1,i}v^r&\text{if }i\neq0,\\
E_{1,n}u^{-1}v^r&\text{if }i=0,
\end{cases}
\end{gather*}
where $E_{i,j}$ is a matrix unit of $\mathfrak{gl}(n)$ whose $(p,q)$ component is $\delta_{p,i}\delta_{q,j}$.
\end{Theorem}

Finally, we consider a presentation of the subalgebra of $A(n)$ generated by $\bigl\{x^+_{i,r}\bigr\}$.
We take a~subspace \smash{$\Omega_+\subset\Omega\bigl(\mathbb{C}\bigl[u^{\pm1},v\bigr]\bigr)/d\mathbb{C}\bigl[u^{\pm1},v\bigr]$} spanned by \smash{$\bigl\{u^{r-1}v^sdu\mid r\geq1,\,s\geq0\bigr\}$}.
We set a~subalgebra $A(n)_+\subset A(n)$ generated by \smash{$\bigoplus_{s\geq0}\bigoplus_{\alpha\in\Delta^+}\mathfrak{g}_\alpha v^s$} and $\Omega_+$. Let us take a grading on~$A(n)$ by \smash{${\rm deg}(h_{i,r})={\rm deg}\bigl(x^\pm_{i,r}\bigr)=r+1$} and denote the graded algebra associated with~$A(n)$~by~$\operatorname{gr}A(n)$.
\begin{Theorem}\label{plus}
The associative algebra $A(n)_+$ is isomorphic to the associative algebra generated by $\bigl\{\bar{x}^+_{i,r}\mid 1\leq i\leq n-1,\,r\geq0\bigr\}$ with the relations:
\begin{gather*}
\bigl[\bar{x}^+_{i,r+1},\bar{x}^+_{j,s}\bigr]=\bigl[\bar{x}^+_{i,r},\bar{x}^+_{j,s+1}\bigr],\\
\bigl[\bar{x}^+_{i,r},\bar{x}^+_{j,s}\bigr]=0 \qquad \text{if }a^A_{i,j}=0,\\
\bigl[\bar{x}^+_{i,r},\bigl[\bar{x}^+_{i,s},\bar{x}^+_{j,u}\bigr]\bigr]=0 \qquad \text{if }a^A_{i,j}=-1.
\end{gather*}
\end{Theorem} 
Let $\mathfrak{A}(n)_{{\rm pre}}$ be the Lie algebra generated by \smash{$\bigl\{h_{i,r},x^\pm_{i,r}\mid 1\leq i\leq n-1,\,r\geq0\bigr\}$} with the defining relations \eqref{Eq2.1}--\eqref{Eq2.8}. We also define $\mathfrak{A}(n)_{{\rm pre},\pm}$ and $\mathfrak{A}(n)_{{\rm pre},0}$ as Lie subalgebras generated by \smash{$\bigl\{x^\pm_{i,r},h_{i,r}\mid 1\leq i\leq n-1,\,r\geq0\bigr\}$}.
Theorem~\ref{plus} follows from the following lemma.
\begin{Lemma}\qquad
\begin{enumerate}\itemsep=0pt
\item[$(1)$] The associative algebra \raisebox{0.3pt}{$U(\mathfrak{A}(n)_{{\rm pre},\pm})$} is isomorphic to a quotient algebra of the free associative algebra generated by \smash{\raisebox{-0.3pt}{$\bigl\{x^\pm_{i,r}\bigr\}$}} divided by \eqref{Eq2.8}.

\item[$(2)$] Let $I_\pm$ be an ideal of $\mathfrak{A}(n)_{{\rm pre}}$ generated by \eqref{Eq2.10}. Then, $I_-\cap A(n)_{{\rm pre},+}=\varnothing$.
\end{enumerate}
\end{Lemma}
\begin{proof} 
\textup{(1)}\ Let $V$ be a quotient algebra of a noncommutative polynomial ring $\mathbb{C}\langle y_{i,d}\mid0\leq i\leq n-1,d\in\mathbb{Z}_{\geq0}\rangle$ divided by the relations $[y_{i,r+1},y_{j,s}]-[y_{i,r},y_{j,s+1}]$. By a direct computation, we obtain a homomorphism
\begin{align*}
\kappa\colon\ \mathfrak{A}(n)_{{\rm pre}}\to\mathfrak{gl}(V)
\end{align*}
by
\begin{align*}
\begin{split}
&\bigl(\kappa\bigl(x^-_{i,d}\bigr)\bigr)(y_{i_1,d_1}\cdots y_{i_r,d_r})=y_{i,d}y_{i_1,d_1}\cdots y_{i_r,d_r},\\
&(\kappa(h_{i,d}))(y_{i_1,d_1}\cdots y_{i_r,d_r})=-\sum_{u=1}^r\limits a^A_{i,i_u}y_{i_1,d_1}\cdots y_{i_{u-1},d_{u-1}}y_{i_u,d+d_u}y_{i_{u+1},d_{u+1}}\cdots y_{i_r,d_r},\\
&\bigl(\kappa\bigl(x^+_{i,d}\bigr)\bigr)(1)=0,\\
&\bigl(\kappa\bigl(x^+_{i,d}\bigr)\bigr)(y_{i_1,d_1}\cdots y_{i_r,d_r})=\delta_{i,i_1}\kappa(h_{i,d+d_1})(y_{i_2,d_2}\cdots y_{i_r,d_r})+y_{i_1,d_1}\bigl(\kappa\bigl(x^+_{i,d}\bigr)\bigr)(y_{i_2,d_2}\cdots y_{i_r,d_r}).
\end{split}
\end{align*}
The homomorphism $\kappa$ induces a surjective homomorphism $\Xi\colon \mathfrak{A}(n)_{{\rm pre},-}\to V$ by $\Xi(x)=(\kappa(x))1$.
By a direct computation, we can also construct a homomorphism
\begin{align*}
\xi\colon\ V\to \mathfrak{A}(n)_{{\rm pre},-},\qquad y_{i,d}\mapsto x^-_{i,d}.
\end{align*}
Since $\xi\circ\Xi$ is an identity map, $\Xi$ is bijective. Then, we find that $U(\mathfrak{A}(n)_{{\rm pre},-})$ is isomorphic to a quotient algebra of the free associative algebra generated by $\{x^-_{i,r}\}$ divided by \eqref{Eq2.8}. Since~$\mathfrak{A}(n)_{{\rm pre},+}$ is isomorphic to $\mathfrak{A}(n)_{{\rm pre},-}$, $U(\mathfrak{A}(n)_{{\rm pre},+})$ is also isomorphic to a quotient algebra of the free associative algebra generated by $\bigl\{x^+_{i,r}\mid r\geq0\bigr\}$ divided by \eqref{Eq2.8}.

\textup{(2)}\ It is enough to show that $I_-$ is contained in $\mathfrak{A}(n)_{{\rm pre},-}$. This follows from the fact that $\ad\bigl(x^+_{i,r}\bigr)$ acts on the left-hand side of \eqref{Eq2.10} trivially.
\end{proof}

\section[A central extension of sl(n){[u\^{}\{pm1\},v]}\^{}omega]{A central extension of $\boldsymbol{\mathfrak{sl}(n)\bigl[u^{\pm1},v\bigr]^\omega}$}\label{sec:3}
First, we set notations about the affine Lie algebra $\widehat{\mathfrak{sl}}(n)$.
Let $\Delta$ \smash{\big(resp.\ $\Delta^+$, $\Delta^-$\big)} be the set of roots (resp.\ positive roots, negative roots) of $\widehat{\mathfrak{sl}}(n)=\mathfrak{sl}(n)\bigl[u^{\pm1}\bigr]\oplus\mathbb{C}c$.
We set an inner product~$(\, ,\, )$ on~\smash{$\widehat{\mathfrak{sl}}(n)$} by
\begin{equation*}
(xu^s,yu^r)=\tr(xy)\delta_{r+s,0},\qquad (xu^r,c)=(c,c)=0.
\end{equation*}
We denote by $\mathfrak{g}_\alpha$ the root space associated with a root $\alpha$.
For $\alpha\in\Delta$, we take \smash{$\bigl\{x_\alpha^k\bigr\}_{k=1}^{\dim (\mathfrak{g}_\alpha)}$} as a~basis of $\mathfrak{g}_\alpha$ such that \smash{$\bigl(x_\alpha^k,x_\beta^l\bigr)=\delta_{\alpha+\beta,0}\delta_{k,l}$} and \smash{$\bigl(x_\alpha^k\bigr)^{\mathsf{T}}=x_{-\alpha}^k$}, where $X^{\mathsf{T}}$ is the transpose of~a~matrix~$X$.
The following relations hold in \smash{$U\bigl(\widehat{\mathfrak{sl}}(n)\bigr)^{\otimes2}$} (see~\cite[Corollary 2.4]{K}):
\begin{gather}
\sum_{k=1}^{\dim \mathfrak{g}_\alpha}\bigl[x_{-\alpha}^k,x_\beta\bigr]\otimes x_\alpha^k=\sum_{k=1}^{\dim \mathfrak{g}_{\beta-\alpha}}x_{\beta-\alpha}^k\otimes \bigl[x_\beta,x_{\alpha-\beta}^k\bigr].\label{out}
\end{gather}
Let $\widehat{\tau}$ be an automorphism of \smash{$U\bigl(\widehat{\mathfrak{sl}}(n)\bigr)$} defined by $X\mapsto-X^{\mathsf{T}}$. Considering the images of both sides of \eqref{out} via $\widehat{\tau}\otimes\id$, we obtain
\begin{gather}
\sum_{k=1}^{\dim \mathfrak{g}_\alpha}\bigl[x_{\alpha}^k,x_{-\beta}\bigr]\otimes x_\alpha^k=-\sum_{k=1}^{\dim \mathfrak{g}_{\beta-\alpha}}x_{\alpha-\beta}^k\otimes \bigl[x_\beta,x_{\alpha-\beta}^k\bigr].\label{out1}
\end{gather}

Let us take an involution of $\mathfrak{sl}(n)\bigl[u^{\pm1},v\bigr]$:
\begin{equation*}
\omega\colon\ \mathfrak{sl}(n)\bigl[u^{\pm1},v\bigr]\to\mathfrak{sl}(n)\bigl[u^{\pm1},v\bigr],\qquad Xu^rv^s\mapsto -(-1)^{s}X^{\mathsf{T}}u^{-r}v^{s}.
\end{equation*}
In this section, we give a central extension of $\mathfrak{sl}(n)\bigl[u^{\pm1},v\bigr]^\omega$.
\begin{Definition}
Let $\mathfrak{BD}(n)$ be the Lie algebra generated by
\begin{gather*}
\{H_{i,2r+1},B_{i,r}\mid0\leq i\leq n-1,\, r\geq0\}
\end{gather*}
with the following relations:
\begin{gather*}
[H_{i,2r+1},H_{j,2s+1}]=0,\\
[H_{i,2r+1},B_{j,s}]=2a^A_{i,j}B_{j,2r+s+1},\\
[B_{i,r+1},B_{j,s}]-[B_{i,r},B_{j,s+1}]=-\delta_{i,j}((-1)^r+(-1)^s)H_{i,r+s},\\
[B_{i,r},B_{j,s}]=0\qquad \text{if }a^A_{i,j}=0,\\
[B_{i,r},[B_{i,s},B_{j,u}]]=-B_{j,r+s+u}\qquad \text{if }a^A_{i,j}=-1.
\end{gather*}
We also denote the universal enveloping algebra of $\mathfrak{BD}(n)$ by $BD(n)$.
\end{Definition}
In~\cite{Gu1}, it was shown that the affine Yangian associated with \smash{$\widehat{\mathfrak{sl}}(n)$} coincides with $A(n)$ if all parameters are zero. If the twisted affine Yangian of types $BD$ exists, we expect that the twisted affine Yangian of types $BD$ should be equal to $BD(n)$ if all the parameters are zero. Moreover, since the twisted affine Yangian should be a coideal of the affine Yangian, it is expected that~$BD(n)$ can be embedded into $A(n)$.

By a direct computation, we obtain a homomorphism $\phi\colon BD(n)\to A(n)$ given by
\begin{align*}
&B_{i,r}=\begin{cases}
E_{i,i+1}v^r-(-1)^rE_{i+1,i}v^r&\text{if }1\leq i\leq 2n-1,\\
E_{n,1}uv^r-(-1)^rE_{1,n}u^{-1}v^r&\text{if }i=0,
\end{cases}\\
&H_{i,2r+1}=\begin{cases}
2(E_{i,i}-E_{i+1,i+1})v^{2r+1}&\text{if }1\leq i\leq 2n-1,\\
2(E_{n,n}-E_{1,1})v^{2r+1}+2u^{-1}v^{2r+1}du&\text{if }i=0.
\end{cases}
\end{align*}
\begin{Theorem}
The homomorphism $\phi$ is injective.
\end{Theorem}
\begin{proof}

For $(i,j,r,s)$ such that $E_{i,j}u^rv^s\in A(n)_+$, we define $E_{i,j}^{r,s}\in BD(n)$ inductively as follows:
\begin{gather*}
E^{0,s}_{i,i+1}=B_{i,s},E^{1,s}_{n,1}=B_{0,s}
\end{gather*}
and
\begin{alignat*}{3}
&E_{i,j}^{r,s}=\bigl[E_{i,j-1}^{r,s},B_{j-1,0}\bigr]\qquad&& \text{if }j>i+1\text{ and }1\leq i<j\leq n,&\\
&E_{i,i+1}^{r,s}=\frac{1}{2}\bigl[H^{r}_{i,s},B_{i,0}\bigr],&\\
&H^r_{i,s}=\bigl[E^{r-1,0}_{i,i+1},E_{i+1,i}^{1,s}\bigr]\qquad&& \text{if }1\leq i\leq n-1,&\\
&H^r_{0,s}=\bigl[B_{0,0},E_{1,n}^{r-1,s}\bigr],&\\
&E_{i,j}^{r,s}=\bigl[E_{i,j-1}^{r,s},B_{j-1,0}\bigr]\qquad&& \text{if }i>j\text{ and }n\geq i\geq j\geq2,&\\
&E_{i,1}^{r,s}=\bigl[E_{i,n}^{r-1,s},B_{0,0}\bigr]\qquad&& \text{if }1\leq i\leq n-1,&\\
&E_{n,1}^{r,s}=\frac{1}{2}\bigl[H^1_{0,0},E_{n,1}^{r-1,s}\bigr]\qquad&&\text{for }r\geq2.&
\end{alignat*}
By a direct computation, we obtain
\begin{alignat*}{3}
&\phi\bigl(E_{i,j}^{r,s}\bigr)=F_{i,j}^{r,s}\qquad&& \text{if }j>i+1\text{ and }1\leq i<j\leq n,&\\
&\phi\bigl(E_{i,i+1}^{r,s}\bigr)=F_{i,i+1}^{r,s},&\\
&\makebox[0pt][l]{$\phi\bigl(H^r_{i,s}\bigr)=F_{i,i}^{r,s}-F_{i+1,i+1}^{r,s}+(r-1)u^{r-1}v^sdu-(r-1)(-1)^su^{-1-r}v^{s}du\qquad\!\! \text{if }1\leq i\leq n-1,$}&\\
&\makebox[0pt][l]{$\phi\bigl(H^r_{0,s}\bigr)=F_{n,n}^{r,s}-F_{1,1}^{r,s}+u^{r-1}v^sdu-(-1)^su^{-1-r}v^{s}du \qquad \text{for }r\geq2,$}&\\
&\phi\bigl(E_{i,j}^{r,s}\bigr)=F_{i,j}^{r,s}\qquad&& \text{if }i>j\text{ and }n\geq i\geq j\geq2,&\\
&\phi\bigl(E_{i,1}^{r,s}\bigr)=F_{i,1}^{r,s}\qquad&& \text{if }1\leq i\leq n-1,&\\
&\phi\bigl(E_{n,1}^{r,s}\bigr)=F_{n,1}^{r,s}\qquad&& \text{for }r\geq2,&
\end{alignat*}
where \smash{$F_{i,j}^{r,s}=E_{i,j}u^rv^s-(-1)^{s}E_{j,i}u^{-r}v^s$}.
Since \smash{$\bigl\{\phi\bigl(E_{i,j}^{r,s}\bigr)\bigr\}\cup\{\phi(H_{i,2r+1})\}$} are linearly independent, the set \smash{$\bigl\{E_{i,j}^{r,s}\bigr\}\cup\{H_{i,2r+1}\}$} is linearly independent.
By the PBW theorem, if we fix an order of the set \smash{$\bigl\{H_{k,2r+1},E_{i,j}^{r,s}\bigr\}$} satisfying that $H_{k,2r+1}<E_{i,j}^{r,s}$, the set of ordered polynomials derived from \smash{$\bigl\{H_{k,2r+1},E_{i,j}^{r,s}\bigr\}$} is linearly independent in $BD(n)$. Thus, it is enough to show that these ordered polynomials span $BD(n)$.

Let us take a grading on $BD(n)$ by ${\rm deg}(B_{j,r})=r+1$ and ${\rm deg}(H_{j,r})=r+1$. We denote the graded algebra associated with $BD(n)$ by $\operatorname{gr}BD(n)$ and the images of \smash{$\bigl\{E^{r,s}_{i,j}\bigr\}$} and \smash{$\bigl\{H^r_{i,s}\bigr\}$} in $\operatorname{gr}BD(n)$ by \smash{$\bigl\{\bar{E}^{r,s}_{i,j}\bigr\}$} and \smash{$\bigl\{\bar{H}^r_{i,s}\bigr\}$}. The graded algebra $\operatorname{gr}BD(n)$ is generated by \smash{$\bigl\{\bar{H}_{i,2r+1},\bar{B}_{i,r}\bigr\}$} which satisfy the following relations:
\begin{gather*}
\bigl[\bar{H}_{i,2r+1},\bar{H}_{j,2s+1}\bigr]=0,\\
\bigl[\bar{H}_{i,2r+1},\bar{B}_{j,s}\bigr]=0,\\
\bigl[\bar{B}_{i,r+1},\bar{B}_{j,s}\bigr]=\bigl[\bar{B}_{i,r},\bar{B}_{j,s+1}\bigr],\\
\bigl[\bar{B}_{i,r},\bar{B}_{j,s}\bigr]=0\qquad \text{if }a^A_{i,j}=0,\\
\bigl[\bar{B}_{i,r},\bigl[\bar{B}_{i,s},\bar{B}_{j,u}\bigr]\bigr]=0\qquad \text{if }a^A_{i,j}=-1.
\end{gather*}
Let $L_+$ be the subalgebra generated by $\bigl\{\bar{B}_{i,r}\bigr\}$.
We can decompose $\operatorname{gr}D(n)$ into the tensor product of $\mathbb{C}[H_{i,2r+1}]$ and $L_+$. By Theorem~\ref{plus}, we have a surjective homomorphism from $A(n)_+$ to $L_+$.
By definition, $\bar{E}^{r,s}_{i,j}$ corresponds to $E_{i,j}u^rv^s$, $H^r_{i,s}$ corresponds to $(E_{i,i}-E_{i+1,i+1})u^rv^s+\allowbreak {(r-1)u^rv^sdu}$ and $H^r_{0,s}$ corresponds to $(E_{n,n}-E_{1,1})u^rv^s+u^{r-1}v^sdu$. Since the set of these elements is a basis of $\bigoplus_{s\geq0}\bigoplus_{\alpha\in\Delta^+}\sum_{k=1}^{\dim \mathfrak{g}_\alpha}x_\alpha^k v^s\oplus\Omega_+$, we find that the set of ordered polynomials derived from \smash{$\bigl\{H_{k,2r+1},E_{i,j}^{r,s}\bigr\}$} spans $BD(n)$. Thus, $\phi$ becomes injective.
\end{proof}

\section{1-parameter twisted affine Yangian}\label{sec:4}
Let us recall the definition of the affine Yangian of type $A$.
\begin{Definition}\label{AAA}
Suppose that $n\geq3$. For $\hbar,\ve\in\mathbb{C}$, the affine Yangian $Y_{\hbar,\ve}\bigl(\widehat{\mathfrak{sl}}(n)\bigr)$ is the associative algebra generated by $X_{i,r}^{+}$, $X_{i,r}^{-}$, $H_{i,r}$, $i \in \{0,1,\dots, n-1\}$, $r = 0,1$, subject to the following defining relations:
\begin{gather*}
[H_{i,r}, H_{j,s}] = 0,\\
\bigl[X_{i,0}^{+}, X_{j,0}^{-}\bigr] = \delta_{i,j} H_{i, 0},\\
\bigl[X_{i,1}^{+}, X_{j,0}^{-}\bigr] = \delta_{i,j} H_{i, 1} = \bigl[X_{i,0}^{+}, X_{j,1}^{-}\bigr],\\
\bigl[H_{i,0}, X_{j,r}^{\pm}\bigr] = \pm a^A_{i,j} X_{j,r}^{\pm},\\
\bigl[\tilde{H}_{i,1}, X_{j,0}^{\pm}\bigr] = \pm a^A_{i,j}\bigl(X_{j,1}^{\pm}\bigr) \qquad \text{if }(i,j)\neq(0,n-1),(n-1,0),\\
\bigl[\tilde{H}_{0,1}, X_{n-1,0}^{\pm}\bigr] = \mp \biggl(X_{n-1,1}^{\pm}+\biggl(\ve+\frac{n}{2}\hbar\biggr) X_{n-1, 0}^{\pm}\biggr),\\
\bigl[\tilde{H}_{n-1,1}, X_{0,0}^{\pm}\bigr] = \mp \biggl(X_{0,1}^{\pm}-\biggl(\ve+\frac{n}{2}\hbar\biggr) X_{0, 0}^{\pm}\biggr),\\
\bigl[X_{i, 1}^{\pm}, X_{j, 0}^{\pm}\bigr] - \bigl[X_{i, 0}^{\pm}, X_{j, 1}^{\pm}\bigr] = \pm a^A_{ij}\frac{\hbar}{2} \bigl\{X_{i, 0}^{\pm}, X_{j, 0}^{\pm}\bigr\}\qquad \text{if }(i,j)\neq(0,n-1),(n-1,0),\\
\bigl[X_{0, 1}^{\pm}, X_{n-1, 0}^{\pm}\bigr] - \bigl[X_{0, 0}^{\pm}, X_{n-1, 1}^{\pm}\bigr]= \mp\frac{\hbar}{2} \bigl\{X_{0, 0}^{\pm}, X_{n-1, 0}^{\pm}\bigr\} + \biggl(\ve+\frac{n}{2}\hbar\biggr) \bigl[X_{0, 0}^{\pm}, X_{n-1, 0}^{\pm}\bigr],\\
\bigl(\ad X_{i,0}^{\pm}\bigr)^{1+|a^A_{i,j}|} \bigl(X_{j,0}^{\pm}\bigr)= 0 \qquad \text{if }i \neq j,
\end{gather*}
where \smash{$\widetilde{H}_{i,1}=H_{i,1}-\frac{\hbar}{2}H_{i,0}^2$} and $\{X,Y\}=XY+YX$. In the case \smash{$\ve=-\frac{n}{2}\hbar$}, we denote \smash{$Y_{\hbar,\ve}\bigl(\widehat{\mathfrak{sl}}(n)\bigr)$} by \smash{$Y_{\hbar}\bigl(\widehat{\mathfrak{sl}}(n)\bigr)$}.
\end{Definition}

Let us recall the standard degreewise completion (see~\cite[Section 1.3 ]{MNT}) of the graded algebra in order to define the elements $J(h_i)$ and~$J\bigl(x^\pm_i\bigr)$, which are tools for constructing an embedding from the twisted affine Yangian of type $BD$ to the affine Yangian of type $A$.
\begin{Definition}
Let $A=\bigoplus_{r\in\mathbb{Z}}A_r$ be the $\mathbb{Z}$-graded algebra.
\begin{enumerate}\itemsep=0pt
\item[$(1)$] The standard degreewise topology of $A$ is the linear topology defined by the sequence $X_{N,d}=\bigoplus_{r\in\mathbb{Z},r\geq N}A_{d-r}A_r$.

\item[$(2)$] Set \smash{$\widetilde{A}_d=\varprojlim_{N} X_nA_d/\overline{X_{N,d}}$},
where \smash{$\overline{X_N}$} is the closure of $X_N$.
We set \smash{$\widetilde{A}=\bigoplus_{d\in\mathbb{Z}}\widetilde{A}_d$} and call~\smash{$\widetilde{A}$} the standard degreewise completion of $A$.
\end{enumerate}
\end{Definition}
We define a degree of \smash{${\rm deg}(H_{i,r})=0,{\rm deg}\bigl(X^\pm_{i,r}\bigr)=\pm\delta_{i,0}$}. We denote the standard degreewise \mbox{completion} of \smash{$Y_{\hbar,\ve}\bigl(\widehat{\mathfrak{sl}}(n)\bigr)$} by \smash{$\widetilde{Y}_{\hbar,\ve}\bigl(\widehat{\mathfrak{sl}}(n)\bigr)$}.

We denote the set of degree $d$ elements by \smash{$Y_{\hbar,\ve}\bigl(\widehat{\mathfrak{sl}}(n)\bigr)_d$}. Set \smash{$A_i\in\widetilde{Y}_{\hbar,\ve}\bigl(\widehat{\mathfrak{sl}}(n)\bigr)$} as
\begin{align*}
A_i={}&\frac{\hbar}{2}\sum_{\substack{s\geq0\\u>i}}\limits E_{u,i}t^{-s}E_{i,u}t^s-\frac{\hbar}{2}\sum_{\substack{s\geq0\\i>v}}\limits E_{i,v}t^{-s}E_{v,i}t^s\\
&{+}\,\frac{\hbar}{2}\sum_{\substack{s\geq0\\u<i}}\limits E_{u,i}t^{-s-1}E_{i,u}t^{s+1}-\frac{\hbar}{2}\sum_{\substack{s\geq0\\i<v}}\limits E_{i,v}t^{-s-1}E_{v,i}t^{s+1}.
\end{align*}
Similarly to~\cite[Section~3]{GNW}, we define
\begin{align*}
J(h_i)&=\widetilde{H}_{i,1}+A_i-A_{i+1}\in \widetilde{Y}_{\hbar,\ve}\bigl(\widehat{\mathfrak{sl}}(n)\bigr),
\end{align*}
where \smash{$\widetilde{H}_{i,1}=H_{i,1}-\frac{\hbar}{2}H_{i,0}^2$} as in Definition~\ref{AAA}.
We also set \smash{$J\bigl(x^\pm_i\bigr)=\pm\frac{1}{2}\bigl[J(h_i),x^\pm_i\bigr]$}.

Guay--Nakajima--Wendlandt~\cite{GNW} defined an automorphism of \smash{$Y_{\hbar,\ve}\bigl(\widehat{\mathfrak{sl}}(n)\bigr)$} by
\begin{equation*}
\tau_i=\exp\bigl(\ad\bigl(X^+_{i,0}\bigr)\bigr)\exp\bigl(-\ad\bigl(X^-_{i,0}\bigr)\bigr)\exp\bigl(\ad\bigl(X^+_{i,0}\bigr)\bigr).
\end{equation*}
Let $\alpha$ be a positive real root of \smash{$\widehat{\mathfrak{sl}}(n)$}. There is an element $w$ of the Weyl group of \smash{$\widehat{\mathfrak{sl}}(n)$} and a~simple root $\alpha_j$ such that $\alpha=w\alpha_j$. Then we define a corresponding root vector by
\begin{equation*}
x^\pm_\alpha=\tau_{i_1}\tau_{i_2}\cdots\tau_{i_{p-1}}\bigl(x^\pm_{j}\bigr),
\end{equation*}
where $w =s_{i_1}s_{i_2}\cdots s_{i_{p-1}}$ is a reduced expression of $w$.
We can define $J\bigl(x^\pm_\alpha\bigr)$ as
\begin{equation*}
J\bigl(x^\pm_\alpha\bigr)=\tau_{i_1}\tau_{i_2}\cdots\tau_{i_{p-1}}J\bigl(x^\pm_{j}\bigr).
\end{equation*}
\begin{Lemma}[{\cite[equation~(3.14) and Proposition 3.21]{GNW}}]\label{J}\qquad
\begin{enumerate}\itemsep=0pt
\item[$(1)$] The following relations hold:
\begin{gather*}
\bigl[J\bigl(h_i\bigr),X^\pm_{j,0}\bigr]=\pm a^A_{i,j}J\bigl(x^\pm_j\bigr)\qquad \text{if } (i,j )\neq (0,n-1 ), (n-1,0 ),\\
\bigl[J\bigl(h_{0}\bigr),X^\pm_{n-1,0}\bigr]=\mp J\bigl(x^\pm_{n-1}\bigr)+\Bigl(\ve+\frac{n}{2}\hbar\Bigr) X_{n-1, 0}^{\pm},\\
\bigl[J\bigl(h_{n-1}\bigr),X^\pm_{0,0}\bigr]=\mp J\bigl(x^\pm_{0}\bigr)-\Bigl(\ve+\frac{n}{2}\hbar\Bigr) X_{0, 0}^{\pm},\\
\bigl[J\bigl(x^\pm_i\bigr), X_{j, 0}^{\pm}\bigr]=\bigl[X_{i, 0}^{\pm}, J\bigl(x^\pm_j\bigr)\bigr]\qquad \text{if } (i,j )\neq (0,n-1 ), (n-1,0 ),\\
\bigl[J\bigl(x^\pm_{0}\bigr), X_{n-1, 0}^{\pm}\bigr]=\bigl[X_{0, 0}^{\pm}, J\bigl(x^\pm_{n-1}\bigr)\bigr]+ \Bigl(\ve+\frac{n}{2}\hbar\Bigr) \bigl[X_{0, 0}^{\pm}, X_{n-1, 0}^{\pm}\bigr].
\end{gather*}
\item[$(2)$]
There exists $c_{\alpha,i}\in\mathbb{C}$ such that
\begin{equation*}
\bigl[J(h_i),x^\pm_\alpha\bigr]=\pm(\alpha_i,\alpha)J\bigl(x_\alpha^\pm\bigr)\pm c_{\alpha,i}x_\alpha^\pm.
\end{equation*}
In particular, $c_{\alpha,i}$ is zero if $\ve+\frac{n}{2}\hbar=0$.
\end{enumerate}
\end{Lemma}
In~\cite{L}, Lu gave the minimalistic presentation for the twisted Yangian. By the same proof as in~\cite[Theorem~3.1]{L}, we can give a minimalistic presentation of $BD(n)$.
\begin{Theorem}
The associative algebra $BD(n)$ is isomorphic to the associative algebra generated~by
\begin{gather*}
\{H_{i,1},B_{i,r}\mid0\leq i\leq n-1,\,r=0,1\}
\end{gather*}
with the following relations:
\begin{gather*}
[H_{i,1},H_{j,1}]=0,\\
[H_{i,1},B_{j,0}]=2a^A_{i,j}B_{j,1},\\
[B_{i,1},B_{j,0}]=[B_{i,0},B_{j,1}]-2\delta_{i,j}H_{i,1},\\
[B_{i,0},B_{j,0}]=0\qquad \text{if }a^A_{i,j}=0,\\
[B_{i,0},[B_{i,0},B_{j,0}]]=-B_{j,0}\qquad \text{if }a^A_{i,j}=-1.
\end{gather*}
\end{Theorem}
By deforming this presentation, we can give a definition of the 1-parameter twisted affine Yangian associated with $\widehat{\mathfrak{so}}(n)$.
\begin{Definition}
For $\hbar\in\mathbb{C}$, set $TY_{\hbar}(\widehat{\mathfrak{so}}(n))$ as the associative algebra generated by
\begin{gather*}
\{H_{i,1},B_{i,r}\mid0\leq i\leq n-1,\,r=0,1\}
\end{gather*}
with the following relations:
\begin{gather*}
[H_{i,1},H_{j,1}]=0,\\
[H_{i,1},B_{j,0}]=2a^A_{i,j}B_{j,1},\\
[B_{i,1},B_{j,0}]=[B_{i,0},B_{j,1}]+\frac{\hbar}{2}a^A_{i,j}\{B_{i,0},B_{j,0}\}-2\delta_{i,j}H_{i,1},\\
[B_{i,0},B_{j,0}]=0\qquad \text{if }a^A_{i,j}=0,\\
[B_{i,0},[B_{i,0},B_{j,0}]]=-B_{j,0}\qquad \text{if }a^A_{i,j}=-1.
\end{gather*}
\end{Definition}
In~\cite{L}, Lu gave the coideal structure of the twisted Yangian. Using the relations~\eqref{out} and~\eqref{out1}, we can give a coideal structure of the twisted affine Yangian $TY_{\hbar}(\widehat{\mathfrak{so}}(n))$ in the same way as in~\cite[Theorem~4.1]{L}. We need to take another filtration for the affine Yangian \smash{$Y_\hbar\bigl(\widehat{\mathfrak{sl}}(n)\bigr)$}. For $d\geq0$, we set \smash{$F_d=\bigoplus_{m\geq d}Y_\hbar\bigl(\widehat{\mathfrak{sl}}(n)\bigr)_m$}. We define \smash{$\widehat{Y}_{\hbar}\bigl(\widehat{\mathfrak{sl}}(n)\bigr)$} as \smash{$\varprojlim_dY_\hbar\bigl(\widehat{\mathfrak{sl}}(n)\bigr)/F_d$}.
\begin{Theorem}\label{coi}
There exists an injective homomorphism
\begin{equation*}
\Phi\colon\ TY_{\hbar}(\widehat{\mathfrak{so}}(n))\to \widehat{Y}_{\hbar}\bigl(\widehat{\mathfrak{sl}}(n)\bigr)
\end{equation*}
given by
\begin{align*}
&\Phi(B_{i,0})=X^+_{i,0}-X^-_{i,0},\\
&\Phi(H_{i,1})=2\widetilde{H}_{i,1}+\hbar\sum_{\alpha\in\Delta^+}\limits\sum_{k=1}^{\dim \mathfrak{g}_\alpha}(\alpha_i,\alpha)x_\alpha^k x_\alpha^k,\\
&\Phi(B_{i,1})=X^+_{i,1}+X^-_{i,1}+\hbar\sum_{\alpha\in\Delta^+}\limits\sum_{k=1}^{\dim \mathfrak{g}_\alpha}\bigl\{\bigl[x^+_{i,0},x_\alpha^k\bigr],x_\alpha^k\bigr\}-\frac{\hbar}{2}\bigl\{x^+_{i,0},h_{i,0}\bigr\},
\end{align*}
where $(\, ,\, )$ is an inner product on \smash{$\sum_{i=0}^n\mathbb{Z}\alpha_i$} defined by $(\alpha_i,\alpha_j)=a_{i,j}$.
\end{Theorem}
Guay--Nakajima--Wendlandt~\cite{GNW} gave the coproduct for the affine Yangian
\begin{equation*}
\Delta\colon\ Y_\hbar\bigl(\widehat{\mathfrak{sl}}(n)\bigr)\to Y_\hbar\bigl(\widehat{\mathfrak{sl}}(n)\bigr)\widehat{\otimes}Y_\hbar\bigl(\widehat{\mathfrak{sl}}(n)\bigr),
\end{equation*}
where \smash{$Y_\hbar\bigl(\widehat{\mathfrak{sl}}(n)\bigr)\widehat{\otimes}Y_\hbar\bigl(\widehat{\mathfrak{sl}}(n)\bigr)$} is the standard degreewise completion of \smash{$\otimes^2Y_\hbar\bigl(\widehat{\mathfrak{sl}}(n)\bigr)$}. Since $\Delta$ preserves the grading, we can extend $\Delta$ to
\begin{equation*}
\widetilde{\Delta}\colon\ \widehat{Y}_{\hbar}\bigl(\widehat{\mathfrak{sl}}(n)\bigr)\to Y_\hbar\bigl(\widehat{\mathfrak{sl}}(n)\bigr)\widehat{\otimes}Y_\hbar\bigl(\widehat{\mathfrak{sl}}(n)\bigr)_{\rm comp},
\end{equation*}
where \smash{$Y_\hbar\bigl(\widehat{\mathfrak{sl}}(n)\bigr)\widehat{\otimes}Y_\hbar\bigl(\widehat{\mathfrak{sl}}(n)\bigr)_{\rm comp}\!$} is defined from \smash{$Y_\hbar\bigl(\widehat{\mathfrak{sl}}(n)\bigr)\widehat{\otimes}Y_\hbar\bigl(\widehat{\mathfrak{sl}}(n)\bigr)$} in the same way as \smash{$\widehat{Y}_{\hbar}\bigl(\widehat{\mathfrak{sl}}(n)\bigr)$}. Then, by the definition, we obtain
\begin{equation*}
\widetilde{\Delta}\bigl(\Phi\bigl(TY_{\hbar}\bigl(\widehat{\mathfrak{so}}(n)\bigr)\bigr)\bigr)\subset \Phi\bigl(TY_{\hbar}\bigl(\widehat{\mathfrak{so}}(n)\bigr)\bigr)\otimes\widehat{Y}_{\hbar}\bigl(\widehat{\mathfrak{sl}}(n)\bigr).
\end{equation*}
Thus, we can consider that $TY_{\hbar}(\widehat{\mathfrak{so}}(n))$ is a coideal of the affine Yangian \smash{$Y_{\hbar}\bigl(\widehat{\mathfrak{sl}}(n)\bigr)$}.

We note that the subalgebra of $TY_{\hbar}(\widehat{\mathfrak{so}}(n))$ generated by $\{B_{i,0}\}$ is not isomorphic to $U(\widehat{\mathfrak{so}}(n))$ but $U\bigl(\mathfrak{sl}(n)\bigl[u^{\pm1},v\bigr]^\omega\bigr)$ since $TY_{\hbar}(\widehat{\mathfrak{so}}(n))$ is a deformation of $\mathfrak{sl}(2n)\bigl[u^{\pm1},v\bigr]^\omega$. By Lemma~\ref{J}, it would be challenging to extend the homomorphism given in Theorem~\ref{coi} to the $2$-parameter affine Yangian if $\ve+\frac{\hbar}{2}n\neq0$. In order to construct a relationship between twisted affine Yangians and $W$-algebras, we need to consider another involution of $\mathfrak{sl}(2n)\bigl[u^{\pm1},v\bigr]$.
\begin{Remark}
In the case $\ve+\frac{\hbar}{2}=0$, there exists a homomorphism from the affine Yangian of type $A$ to the rectangular $W$-algebra of type $A$ associated with a critical level. Unfortunately, in the case $\ve+\frac{\hbar}{2}=0$, this homomorphism is not surjective so that we can not expect the corresponding result to~\cite{BK}. Similarly, it is expected that $BD(n)$ corresponds to the rectangular $W$-algebras of types $BCD$ at the critical level and this homomorphism will not be surjective.
\end{Remark}
\section[A central extension of sl(2n){[u\^{}\{pm1\},v]}\^{}tau]{A central extension of $\boldsymbol{\mathfrak{sl}(2n)\bigl[u^{\pm1},v\bigr]^\tau}$}\label{sec:5}
We give another involution of $\mathfrak{sl}(2n)$. Set $I_n=\{\pm1,\pm2,\dots,\pm n\}$ and take $\{E_{i,j}\}_{i,j\in I_n}$ as the set of matrix units in $\mathfrak{gl}(2n)$. We also define $\tau$ as an automorphism of $\mathfrak{gl}(2n)\bigl[u^{\pm1},v\bigr]$ by $\tau(E_{i,j}u^rv^s)=-(-1)^sE_{-j,-i}u^rv^s$. By definition of $\tau$, $\mathfrak{sl}(2n)\bigl[u^{\pm1},v\bigr]$ is preserved by $\tau$. In the next section, we give a central extension of $\mathfrak{sl}(2n)\bigl[u^{\pm1},v\bigr]^\tau$.

For $r\geq0$, we define matrices \smash{$\bigl(a^{2r}_{i,j}\bigr)_{0\leq i\leq n,0\leq j\leq n}$} and \smash{$\bigl(a^{2r+1}_{i,j}\bigr)_{1\leq i\leq n-1,0\leq j\leq n}$} as follows:
\begin{align*}
a^{2r}_{i,j}&=\begin{cases}
2\delta_{i,j}-\delta_{i+1,j}-\delta_{i,j+1}&\text{if }3\leq i\leq n-3,\\
2\delta_{j,n-2}-\delta_{j,n-3}-\delta_{j,n-1}-\delta_{j,n}&\text{if }i=n-2,\\
2\delta_{j,n-1}-\delta_{j,n-2}&\text{if }i=n-1,\\
2\delta_{j,n}-\delta_{j,n-2}&\text{if }i=n,\\
2\delta_{j,2}-\delta_{j,0}-\delta_{j,1}-\delta_{j,3}&\text{if }i=2,\\
2\delta_{j,1}-\delta_{j,2}&\text{if }i=1,\\
2\delta_{j,0}-\delta_{j,2}&\text{if }i=0,
\end{cases}\\
a^{2r+1}_{i,j}&=\begin{cases}
a^{2r}_{i,j}&\text{if }2\leq i\leq n-2,\\
2\delta_{j,n-1}+2\delta_{j,n}-\delta_{j,n-2}&\text{if }i=n-1,\\
2\delta_{j,1}+2\delta_{j,0}-\delta_{j,2}&\text{if }i=1.
\end{cases}
\end{align*}
We note that \smash{$\bigl(a^{2r}_{i,j}\bigr)_{i,j=1}^n$} is the Cartan matrix of $\widehat{\mathfrak{so}}(2n)$.
\begin{Definition}
We define $\mathfrak{D}(2n)$ as the Lie algebra generated by
\begin{equation*}
\bigl\{H_{i,r},X^\pm_{i,r},H_{n,2r},H_{0,2r},X^\pm_{-1,2r+1},X^\pm_{n+1,2r+1} \mid 1\leq i\leq n-1,\, r\in\mathbb{Z}_{\geq0}\bigr\}
\end{equation*}
with the following commutator relations:
\begin{gather}
[H_{i,r},H_{j,s}]=0, \label{5111}\\
\bigl[H_{i,r},X^\pm_{j,s}\bigr]=\pm a^r_{i,j}X^\pm_{j,r+s}\qquad \text{for }0\leq j\leq n,\label{5112}\\
\bigl[X^+_{i,r},X^-_{j,s}\bigr]=\begin{cases}
\delta_{i,j}H_{i,r+s}& \text{if} \ 0\leq i,j\leq n,\\
& \hphantom{\text{for}}\! (i,j)\neq(n-1,n),(n,n-1),(0,1),(1,0),(n,n),(0,0),\\
H_{n,r+s}& \text{if }i=j=n\text{ and }r+s\text{ is }even,\\
H_{n-1,r+s}& \text{if }i=j=n\text{ and }r+s\text{ is }odd,\\
H_{0,r+s}& \text{if }i=j=0\text{ and }r+s\text{ is }even,\\
H_{1,r+s}& \text{if }i=j=0\text{ and }\text{ and }r+s\text{ is }odd,
\end{cases}\!\!\! \label{5113}\\
\bigl[X^+_{n-1,r},X^-_{n,s}\bigr]=-\delta_{r+s,{\rm odd}}X^-_{n+1,r+s},\qquad \bigl[X^+_{n,r},X^-_{n-1,s}\bigr]=-\delta_{r+s,{\rm odd}}X^+_{n+1,r+s},\label{5114}\\
\bigl[X^+_{0,r},X^-_{1,s}\bigr]=-\delta_{r+s,{\rm odd}}X^+_{-1,r+s},\qquad \bigl[X^+_{1,r},X^-_{0,s}\bigr]=-\delta_{r+s,{\rm odd}}X^-_{-1,r+s}, \label{5115}\\
\bigl[X^\pm_{i,r},X^+_{n+1,s}\bigr]=0,\qquad\bigl[X^\pm_{i,r},X^-_{n+1,s}\bigr]=0 \qquad\text{if }1\leq i\leq n-1, \label{5116}\\
\bigl[X^\pm_{i,r},X^+_{-1,s}\bigr]=0,\qquad\bigl[X^\pm_{i,r},X^-_{-1,s}\bigr]=0 \qquad\text{if }2\leq i\leq n+1,\label{5117}\\
\bigl[X^+_{n-1,r},X^+_{n+1,s}\bigr]=2X^+_{n,r+s},\qquad \bigl[X^-_{n-1,r},X^+_{n+1,s}\bigr]=0,\label{5118}\\
\bigl[X^+_{n-1,r},X^-_{n+1,s}\bigr]=0,\qquad \bigl[X^-_{n-1,r},X^-_{n+1,s}\bigr]=-2X^-_{n,r+s},\label{5119}\\
\bigl[X^+_{n,r},X^+_{n+1,s}\bigr]=0,\qquad \bigl[X^-_{n,r},X^+_{n+1,s}\bigr]=-2X^{r+s}_{n-1,r+s}, \label{5120}\\
\bigl[X^+_{n,r},X^-_{n+1,s}\bigr]=2X^+_{n-1,r+s},\qquad \bigl[X^-_{n,r},X^-_{n+1,s}\bigr]=0, \label{5121}\\
\bigl[X^+_{0,r},X^+_{-1,s}\bigr]=0,\qquad \bigl[X^-_{0,r},X^+_{-1,s}\bigr]=-2X^-_{1,r+s}, \label{5122}\\
\bigl[X^+_{0,r},X^-_{-1,s}\bigr]=2X^+_{1,r+s},\qquad \bigl[X^-_{0,r},X^-_{-1,s}\bigr]=0, \label{5123}\\
\bigl[X^+_{1,r},X^+_{-1,s}\bigr]=2X^+_{0,r+s},\qquad \bigl[X^-_{1,r},X^+_{-1,s}\bigr]=0, \label{5124}\\
\bigl[X^+_{1,r},X^-_{-1,s}\bigr]=0,\qquad \bigl[X^-_{1,r},X^-_{-1,s}\bigr]=-2X^-_{0,r+s}, \label{5125}\\
\bigl[X^\pm_{i,r},X^\pm_{i,s}\bigr]=0, \label{5126}\\
\bigl[X^\pm_{i,r+1},X^\pm_{j,s}\bigr]=\bigl[X^\pm_{i,r},X^\pm_{j,s+1}\bigr]\quad \text{if }0\leq i\leq j\leq n,(i,j)\neq(0,1),(n-1,n),\label{5127}\\
\bigl[X^\pm_{0,r+1},X^\pm_{1,s}\bigr]=-\bigl[X^\pm_{0,r},X^\pm_{1,s+1}\bigr],\qquad \bigl[X^\pm_{n-1,r+1},X^\pm_{n,s}\bigr]=-\bigl[X^\pm_{n-1,r},X^\pm_{n,s+1}\bigr], \label{5128}\\
\bigl[X^\pm_{i,r},X^\pm_{j,s}\bigr]=0\quad \text{if }a^0_{i,j}=0,0\leq i\neq j\leq n,(i,j)\neq(0,1),(1,0),(n-1,n),(n,n-1),\!\!\!\label{5129}\\
\bigl[X^\pm_{i,r},\bigl[X^\pm_{i,s},X^\pm_{j,u}\bigr]\bigr]=0\qquad \text{if }a^0_{i,j}=-1,\label{5130}\\
\bigl[X^\pm_{n-1,0},\bigl[X^\pm_{n-1,0},X^\pm_{n,2r+1}\bigr]\bigr]=0,\label{5131}\\
\bigl[X^\pm_{n,0},\bigl[X^\pm_{n,0},X^\pm_{n-1,2r+1}\bigr]\bigr]=0,\label{5132}\\
\bigl[X^\pm_{1,0},\bigl[X^\pm_{1,0},X^\pm_{0,2r+1}\bigr]\bigr]=0,\label{5133}\\
\bigl[X^\pm_{0,0},\bigl[X^\pm_{0,0},X^\pm_{1,2r+1}\bigr]\bigr]=0,\label{5134}\\
\bigl[X^+_{1,1},\bigl[X^-_{1,0},(E_{1,1}-E_{2,2})t^u\bigr]\bigr]-\bigl[X^+_{1,0}, \bigl[X^-_{1,1},(E_{1,1}-E_{2,2})t^u\bigr]\bigr]=0. \label{5135}
\end{gather}
We also denote the universal enveloping algebra of $\mathfrak{D}(2n)$ by $D(2n)$.
\end{Definition}
Set \smash{$f^{r,s}_{i,j}=(E_{i,j}-(-1)^sE_{-j,-i})u^rv^s\in\mathfrak{sl}(2n)\bigl[u^{\pm1},v\bigr]$}. Then, \smash{$\mathfrak{sl}(2n)\bigl[u^{\pm1},v\bigr]^\tau$} is spanned by \smash{$\bigl\{f^{r,s}_{i,j}\mid i,j\in I_n,\,r\in\mathbb{Z},\,s\in\mathbb{Z}_{\geq0}\bigr\}$}. By a direct computation, we can construct a homomorphism \smash{$\pi\colon \mathfrak{D}(2n)\to\mathfrak{sl}(2n)\bigl[u^{\pm1},v\bigr]^\tau$} determined by
\begin{gather*}
H_{i,r}=\begin{cases}
f^r_{i,i}v^r-f^r_{i+1,i+1}v^r&\text{if }1\leq i\leq n-1,\\
f^r_{n-1,n-1}v^r+f^r_{n,n}v^r&\text{if }i=n\text{ and $r$ is even},\\
-\bigl(f^r_{1,1}v^r+f^r_{2,2}v^r\bigr)&\text{if }i=0\text{ and $r$ is even},
\end{cases}\\
X^+_{i,r}=\begin{cases}
f^r_{i,i+1}v^r&\text{if }1\leq i\leq n-1,\\
f^r_{n-1,-n}v^r&\text{if }i=n,\\
(-1)^rf^r_{-2,1}uv^r&\text{if }i=0,\\
f^r_{n,-n}v^r&\text{if }i=n+1\text{ and $r$ is odd},\\
f^r_{-1,1}uv^r&\text{if }i=-1\text{ and $r$ is odd},
\end{cases}\\
X^-_{i,r}=\begin{cases}
f^r_{i+1,i}v^r&\text{if }1\leq i\leq n-1,\\
f^r_{-n,n-1}v^r&\text{if }i=n,\\
(-1)^rf^r_{1,-2}u^{-1}v^r&\text{if }i=0,\\
f^r_{-n,n}v^r&\text{if }i=n+1\text{ and $r$ is odd},\\
f^r_{1,-1}u^{-1}v^r&\text{if }i=-1\text{ and $r$ is odd}.
\end{cases}
\end{gather*}
Next, we show that $\mathfrak{D}(2n)$ is a central extension of $\mathfrak{sl}(2n)\bigl[u^{\pm1},v\bigr]^\tau$.

Let $\{\alpha_i\mid0\leq i\leq n\}$ be the set of simple roots of $\widehat{\mathfrak{so}}(2n)$, let $W$ be the Weyl group of $\widehat{\mathfrak{so}}(2n)$, and let~$\Delta$ be the set of roots of $\widehat{\mathfrak{so}}(2n)$. We also denote
\begin{gather*}
Q=\sum_{i=0}^n \mathbb{Z}\alpha_i,Q_+=\sum_{i=0}^n \mathbb{Z}_{\geq0}\alpha_i\setminus\{0\},\qquad Q_-=-Q_+\qquad \text{and}\\ \delta=\alpha_0+\alpha_1+\alpha_{n-1}+\alpha_n+\sum_{j=2}^{n-2} \alpha_j.
\end{gather*} We also take a subset of $Q$ as follows:
\begin{align*}
\Delta^{{\rm ex}}=\Delta\cup\Biggl\{&\pm\Bigg(2\sum_{j=i}^{n-2}\alpha_j+\alpha_{n-1}+\alpha_n\Biggr)+s\delta,\\
&\pm(\alpha_{n-1}-\alpha_n)+s\delta,\pm(\alpha_0-\alpha_1)+s\delta\mid 1\leq i\leq n-2,\,s\in\mathbb{Z}\Biggr\}.
\end{align*}
We define a degree of $\mathfrak{D}(2n)$ by $Q\times\mathbb{Z}_{\geq0}$ by assigning generators as follows:
\begin{gather*}
{\rm deg}(H_{i,r})=(0,r),\qquad {\rm deg}(X^\pm_{i,r})=\begin{cases}
(\pm\alpha_i,r)&\text{if }0\leq i\leq n,\\
(\pm(\alpha_{n}-\alpha_{n-1}),r)&\text{if }i=n+1,\\
(\pm(\alpha_{0}-\alpha_1),r)&\text{if }i=0.
\end{cases}
\end{gather*}
We denote the space of elements of degree $(\alpha,k)$ in $\mathfrak{D}(2n)$ by \smash{$\mathfrak{D}(2n)_\alpha^k$}. The following lemma can be proved in the same way as~\cite[Lemma 3.1]{MRY}.
\begin{Lemma}\label{Lem128}\qquad
\begin{enumerate}\itemsep=0pt
\item[$(1)$] Let $\mathfrak{s}_\pm$ be the subspace of $\mathfrak{D}(2n)$ spanned by \smash{$\bigl\{\prod_{i=1}^v\ad\bigl(X^\pm_{u_i,r_i}\bigr)X^\pm_{u_{v+1},r_{v+1}}\bigr\}$}
and~$\mathfrak{h}$ as the space spanned by $\{H_{i,r}\}$. Then, $\mathfrak{D}(2n)$ can be decomposed to
\begin{align*}
\mathfrak{s}^+& \oplus\mathfrak{h}\oplus\mathfrak{s}^-\oplus \bigoplus_{r\in\mathbb{Z}_{\geq0}}\mathbb{C}X^+_{n+1,2r+1}\oplus \bigoplus_{r\in\mathbb{Z}_{\geq0}}\mathbb{C}X^-_{n+1,2r+1}\\
& \oplus \bigoplus_{r\in\mathbb{Z}_{\geq0}}\mathbb{C}X^+_{-1,2r+1}\oplus \bigoplus_{r\in\mathbb{Z}_{\geq0}}\mathbb{C}X^-_{-1,2r+1}.
\end{align*}
\item[$(2)$] If $\alpha\in Q_{\pm}$, $\mathfrak{D}(2n)_\alpha^k$ is spanned by
\begin{equation*}
\Biggl\{\prod_{i=1}^v\ad\bigl(X^\pm_{u_i,r_i}\bigr)X^\pm_{u_{v+1},r_{v+1}}\:\bigg|\: \sum_{i=1}^{v+1}\alpha_{u_i}=\alpha,\,\sum_{i=1}^{v+1}r_i=k\Biggr\}.
\end{equation*}
\end{enumerate}
\end{Lemma}
For $0\leq i\leq n$, we set
\[ s_i\in W \qquad \text{ as } s_i(\alpha_j)=\begin{cases}
-\alpha_i&\text{if }i=j,\\
\alpha_i+\alpha_j&\text{if }a^0_{i,j}=-1,\\
\alpha_j&\text{otherwise.}
\end{cases}
\]
\begin{Lemma}
By the action of the Weyl group, $\Delta^{{\rm ex}}$ is generated by $\{\alpha_j\mid0\leq j\leq n\}$ and $\pm(\alpha_{n-1}-\alpha_n)$, $\pm(\alpha_0-\alpha_1)$.
\end{Lemma}
\begin{proof}
Any root of $\Delta$ can be derived from $\{\alpha_j\mid0\leq j\leq n\}$ by the Weyl group. Thus, it is enough to show that \smash{$\pm\bigl(2\sum_{j=i}^{n-2}\alpha_j+\alpha_{n-1}+\alpha_n\bigr)+s\delta$} and $\pm(\alpha_{n-1}-\alpha_n)+s\delta$ can be derived from~$\pm(\alpha_{n-1}-\alpha_n)$,~$\pm(\alpha_0-\alpha_1)$.

First, we consider the case where $s$ is even.
By the definition of $s_i$, we obtain the following relations:
\begin{align*}
&s_n(\alpha_{n-1}-\alpha_n+s\delta)=\alpha_{n-1}+\alpha_n+s\delta,\\
&s_{n-1}(\alpha_{n-1}+\alpha_n+s\delta)=-\alpha_{n-1}+\alpha_n+s\delta,\\
&s_{i-1}\Biggl(2\sum_{j=i}^{n-2}\alpha_j+\alpha_{n-1}+\alpha_n+s\delta\Biggr)=2\sum_{j=i-1}^{n-2}\alpha_j+\alpha_{n-1}+\alpha_n+s\delta,\\
&s_{0}\Biggl(2\sum_{j=1}^{n-2}\alpha_j+\alpha_{n-1}+\alpha_n+s\delta\Biggr)=-\alpha_{0}-\alpha_1+(s+1)\delta,\\
&s_0(-\alpha_{0}-\alpha_1+(s+1)\delta)=\alpha_{0}-\alpha_1+(s+1)\delta,\\
&s_1(\alpha_0-\alpha_1+s\delta)=\alpha_{0}+\alpha_1+s\delta=-2\sum_{j=1}^{n-2}\alpha_j-\alpha_{n-1}-\alpha_n+(s+1)\delta,\\
&s_{i}\Biggl(-2\sum_{j=i}^{n-2}\alpha_j-\alpha_{n-1}-\alpha_n+(s+1)\delta\Biggr)=-2\sum_{j=i+1}^{n-2}\alpha_j-\alpha_{n-1}-\alpha_n+(s+1)\delta,\\
&s_{n-1}(-\alpha_{n-1}-\alpha_n+(s+1)\delta)=\alpha_{n-1}-\alpha_n+(s+1)\delta.
\end{align*}
Then, by these relations, \smash{$\bigl(2\sum_{j=i}^{n-2}\alpha_j+\alpha_{n-1}+\alpha_n\bigr)+s\delta$}, \smash{$-\bigl(2\sum_{j=i}^{n-2}\alpha_j+\alpha_{n-1}+\alpha_n\bigr)+(s+1)\delta$}, $\pm(\alpha_{n-1}-\alpha_n)+s\delta$ and $\pm(\alpha_0-\alpha_1)+s\delta$ can be derived from $\alpha_{n-1}-\alpha_n$ and $\alpha_0-\alpha_1$ for $s\geq0$.
Similarly to the positive case, $\bigl(2\sum_{j=i}^{n-2}\alpha_j+\alpha_{n-1}+\alpha_n\bigr)-(s+1)\delta$, $-\bigl(2\sum_{j=i}^{n-2}\alpha_j+\alpha_{n-1}+\alpha_n\bigr)-s\delta$, $\pm(\alpha_{n-1}-\alpha_n)-(s+1)\delta$ and $\pm(\alpha_0-\alpha_1)-(s+1)\delta$ can be derived from $-(\alpha_{n-1}-\alpha_n)$ and $-(\alpha_0-\alpha_1)$ for~${s\geq0}$.
\end{proof}
By the relations \eqref{5129}--\eqref{5134}, we can define an automorphism of $D(2n)$ by
\begin{gather*}
\tau_i=\exp\bigl(\ad\bigl(X^+_{i,0}\bigr)\bigr)\exp\bigl(-\ad\bigl(X^-_{i,0}\bigr)\bigr)\exp\bigl(\ad\bigl(X^+_{i,0}\bigr)\bigr).
\end{gather*}
By a direct computation, we obtain the following lemma.
\begin{Lemma}\label{Weyl}
The following relations hold for $1\leq i\leq n-2$:
\begin{gather*}
\tau_i(X^+_{j,r})=\begin{cases}
-X^-_{i,r}&\text{if }i=j,\\
\bigl[X^+_{i,0},X^+_{j,r}\bigr]&\text{if }a^0_{i,j}=-1,\\
(-1)^rX^+_{j,r}&\text{if }(i,j)=(0,1),(1,0),(n-1,n),(n,n-1),\\
X^+_{i,r}&\text{otherwise},
\end{cases}\\
\tau_i(X^-_{j,r})=\begin{cases}
-X^+_{i,r}&\text{if }i=j,\\
-\bigl[X^-_{i,0},X^-_{j,r}\bigr]&\text{if }a^0_{i,j}=-1,\\
(-1)^rX^-_{j,r}&\text{if }(i,j)\neq(0,1),(1,0),(n-1,n),(n,n-1),\\
X^+_{j,r}&\text{otherwise}.
\end{cases}
\end{gather*}
\end{Lemma}
\begin{proof}
We only consider the case where $(i,j)=(n-1,n)$ when $r$ is odd. The other cases can be proved in a similar way. First, we prove the $+$ case. By \eqref{5131}, we have
\begin{align*}
&\exp\bigl(\ad\bigl(X^+_{n-1,0}\bigr)\bigr)X^+_{n,2r+1}=X^+_{n,2r+1}+\bigl[X^+_{n-1,0},X^+_{n,2r+1}\bigr]=X^+_{n,2r+1}+\bigl[X^+_{n-1,0},X^+_{n,2r+1}\bigr].
\end{align*}
By \eqref{5114}, \eqref{5112} and \eqref{5118}, we obtain
\begin{align*}
&-\ad\bigl(X^-_{n-1,0}\bigr)\bigl(X^+_{n,2r+1}+\bigl[X^+_{n-1,0},X^+_{n,2r+1}\bigr]\bigr)\\
&\qquad{}=-\bigl[X^-_{n-1,0},X^+_{n,2r+1}\bigr]-\bigl[X^-_{n-1,0},\bigl[X^+_{n-1,0},X^+_{n,2r+1}\bigr]\bigr]\\
&\qquad{}=-X^+_{n+1,2r+1}+\bigl[h_{n-1,0},X^+_{n,2r+1}\bigr]-\bigl[X^+_{n-1,0},\bigl[X^-_{n-1,0},X^+_{n,2r+1}\bigr]\bigr]\\
&\qquad{}=-X^+_{n+1,2r+1}-\bigl[X^+_{n-1,0},X^+_{n+1,2r+1}\bigr]=X^+_{n+1,2r+1}-2X^+_{n,2r+1}.
\end{align*}
By \eqref{5118} and \eqref{5114}, we have
\begin{align*}
-\frac{1}{2}\ad\bigl(X^-_{n-1,0}\bigr)\bigl(-X^+_{n+1,2r+1}-2X^+_{n,2r+1}\bigr)
&=0+X^+_{n+1,2r+1}.
\end{align*}
Thus, we obtain
\begin{align*}
\exp\bigl(-\ad\bigl(X^-_{n-1,0}\bigr)\bigr)\exp\bigl(\ad\bigl(X^+_{n-1,0}\bigr)\bigr)X^+_{n,2r+1}&=-X^+_{n,2r+1}+\bigl[X^+_{n-1,0},X^+_{n,2r+1}\bigr].
\end{align*}
By \eqref{5131}, we obtain
\begin{align*}
&\ad\bigl(X^+_{n-1,0}\bigr)\bigl(-X^+_{n,2r+1}+\bigl[X^+_{n-1,0},X^+_{n,2r+1}\bigr]\bigr)\\
&\qquad{}=-\bigl[X^+_{n-1,0},X^+_{n,2r+1}\bigr]+\bigl[X^+_{n-1,0},\bigl[X^+_{n-1,0},X^+_{n,2r+1}\bigr]\bigr] =-\bigl[X^+_{n-1,0},X^+_{n,2r+1}\bigr]+0,
\\
& \ad\bigl(X^+_{n-1,0}\bigr)\bigl(-[X^+_{n-1,0},X^+_{n,2r+1}]\bigr)=0.
\end{align*}
Thus, we have obtained $\tau_{n-1}\bigl(X^+_{n,2r+1}\bigr)=-X^+_{n,2r+1}$.

Next, we show the $-$ case. By \eqref{5114} and \eqref{5119}, we have
\begin{align*}
\exp\bigl(\ad\bigl(X^+_{n-1,0}\bigr)\bigr)X^-_{n,2r+1}&=X^-_{n,2r+1}+\bigl[X^+_{n-1,0},X^-_{n,2r+1}\bigr]+\frac{1}{2}\bigl[X^+_{n-1,0},\bigl[X^+_{n-1,0},X^-_{n,2r+1}\bigr]\bigr]\\
&=X^-_{n,2r+1}-X^-_{n+1,2r+1}+0.
\end{align*}
By \eqref{5119}, we obtain
\begin{align*}
\bigl(-\ad\bigl(X^-_{n-1,0}\bigr)\bigr)\bigl(X^-_{n,2r+1}-X^-_{n+1,2r+1}\bigr)&=-\bigl[X^-_{n-1,0},X^-_{n,2r+1}\bigr]+\bigl[X^-_{n-1,0},X^-_{n+1,2r+1}\bigr]\\
&=-\bigl[X^-_{n-1,0},X^-_{n,2r+1}\bigr]-2X^-_{n,2r+1}.
\end{align*}
By \eqref{5131}, we have
\begin{align*}
-\frac{1}{2}\ad\bigl(X^-_{n-1,0}\bigr)\bigl(-\bigl[X^-_{n-1,0},X^-_{n,2r+1}\bigr]-2X^-_{n,2r+1}\bigr)=0+\bigl[X^-_{n-1,0},X^-_{n,2r+1}\bigr].
\end{align*}
Then, we obtain
\begin{align*}
\exp\bigl(-\ad\bigl(X^-_{n-1,0}\bigr)\bigr)\exp\bigl(\ad\bigl(X^+_{n-1,0}\bigr)\bigr)X^-_{n,2r+1} =-X^-_{n,2r+1}-X^-_{n+1,2r+1}.
\end{align*}
By \eqref{5114} and \eqref{5119}, we obtain
\begin{align*}
& \ad\bigl(X^+_{n-1,0}\bigr)\bigl(-X^-_{n,2r+1}-X^-_{n+1,2r+1}\bigr) =X^-_{n+1,2r+1}-0,
\\
& \frac{1}{2}\bigl(\ad\bigl(X^+_{n-1,0}\bigr)\bigr)\bigl(X^-_{n+1,2r+1}\bigr)=0.
\end{align*}
Thus, we have obtained $\tau_{n-1}\bigl(X^-_{n,2r+1}\bigr)=-X^-_{n,2r+1}$.
\end{proof}
\begin{Corollary}\label{Cor25}
The following relations hold:{\samepage
\begin{align*}
\tau_i(f_{j,k}^{r,s})&=\begin{cases}
-f_{k,j}^{r,s}&\text{if }(j,k)=(i,i+1),(i+1,i),\\
(-1)^rf_{j,k}^{r,s}&\text{if }(j,k)=(i,-(i+1)),(-(i+1),i),(-(i+1),i),(-i,(i+1)),\\
-f_{j,k}^{r,s}&\text{if }j=i,k\neq\pm(i+1),\\
-f_{-i-1,k}^{r,s}&\text{if }j=-i,k\neq(i+1),\\
f_{i,k}^{r,s}&\text{if }j=i+1,k\neq \pm i,\\
f_{-i,k}^{r,s}&\text{if }j=-(i+1)\neq \pm i,\\
-f_{j,i+1}^{r,s}&\text{if }k=i,j\neq\pm(i+1),\\
-f_{j,-i-1}^{r,s}&\text{if }k=-i,j\neq(i+1),\\
f_{j,i}^{r,s}&\text{if }k=i+1,j\neq i,\\
f_{j,-i}^{r,s}&\text{if }k=-(i+1),j\neq i,\\
f_{j,k}^{r,s}&\text{otherwise},
\end{cases}\\
\tau_n(f_{j,k}^{r,s})&=\begin{cases}
(-1)^rf_{j,k}^{r,s}&\text{if }(j,k)=(n-1,n),(n,n-1),\\
-f_{k,j}^{r,s}&\text{if }(j,k)=(n-1,-n),(n,-(n-1)),(-(n-1),n),\\
-f_{-n,k}^{r,s}&\text{if }j=n-1,k\neq\pm n,\\
-f_{n,k}^{r,s}&\text{if }j=-(n-1),k\neq n,\\
f_{-(n-1),k}&\text{if }j=n,k\neq \pm n-1,\\
f_{(n-1),k}^{r,s}&\text{if }j=-n\neq \pm n-1,\\
-f_{j,-n}^{r,s}&\text{if }k=n-1,j\neq\pm n,\\
-f_{j,n}^{r,s}&\text{if }k=-(n-1),j\neq n,\\
f_{j,-(n-1)}^{r,s}&\text{if }k=n,j\neq n-1,\\
f_{j,(n-1)}^{r,s}&\text{if }k=-n,j\neq n-1,\\
f_{j,k}^{r,s}&\text{otherwise},
\end{cases}\\
\tau_0(f_{j,k}^{r,s})&=\begin{cases}
(-1)^rf_{j,k}^{r,s}&\text{if }(j,k)=(1,2),(j,k)=(2,1),\\
-f_{k,j}^{r-2,s}&\text{if }(j,k)=(1,-2),(1,-2),\\
-f_{k,j}^{r+2,s}&\text{if }(j,k)=(-2,1),(-1,2),\\
f_{-2,j}^{r+2,s}&\text{if }j=1,k\neq\pm 2,\\
f_{2,k}^{r-2,s}&\text{if }j=-1,k\neq 2,\\
-f_{-1,k}^{r+1,s}&\text{if }j=2,k\neq \pm 1,\\
-f_{1,k}^{r-1,s}&\text{if }j=-2\neq 1,\\
f_{j,-2}^{r-1,s}&\text{if }k=1,j\neq\pm 2,\\
f_{j,2}^{r+1,s}&\text{if }k=-1,j\neq 2,\\
-f_{j,-1}^{r-1,s}&\text{if }k=2,j\neq 1,\\
-f_{j,1}^{r+1,s}&\text{if }k=-2,j\neq 1,\\
f_{j,k}^{r,s}&\text{otherwise},
\end{cases}
\end{align*}
for $j\neq \pm k$ and either $j$ or $k$ is positive.}
\end{Corollary}

By Lemma~\ref{Weyl}, we find that $\tau_i\bigl(\mathfrak{D}(2n)_\alpha^k\bigr)$ is contained in $\mathfrak{D}(2n)_{s_i(\alpha)}^k$
In the same way as~\cite[Proposition~3.2]{MRY}, we obtain the following theorem.
\begin{Theorem}\label{Thm128}
We find that
\[\begin{cases}
\dim \mathfrak{D}(2n)_\alpha^k=1&\text{if }\alpha\in\Delta^{{\rm ex}}\setminus\{k\delta\mid k\in\mathbb{Z}\setminus\{0\}\},\\
\dim \mathfrak{D}(2n)_\alpha^k=0&\text{if }\alpha\not\in\Delta^{{\rm ex}}.
\end{cases}
\]
\end{Theorem}
By Theorem~\ref{Thm128}, the kernel of $\pi$ is contained in $\bigoplus_{r\in\mathbb{Z}} \mathfrak{D}(2n)_{r\delta}^k$. Since $\bigl[X^\pm_{i,0},\mathfrak{D}(2n)_{r\delta}^k\bigr]\cap \mathfrak{D}(2n)_{r\delta}^k$ is empty by Lemma~\ref{Lem128}, the kernel of $\pi$ is contained in the center of $D(2n)$.
Thus, we find that $\mathfrak{D}(2n)$ is a central extension of $\mathfrak{sl}(2n)\bigl[u^{\pm1},v\bigr]^\tau$.
\section[Embedding of D(2n) into A(2n)]{Embedding of $\boldsymbol{D(2n)}$ into $\boldsymbol{A(2n)}$}\label{sec:6}
In this section, we show that the associative algebra $\mathfrak{D}(2n)$ is a subalgebra of the associative algebra $A(2n)$.
By a direct computation, we can construct a homomorphism
$\phi\colon \mathfrak{D}(2n)\to A(2n)$
defined by
\begin{gather*}
H_{i,r}=\begin{cases}
f^r_{i,i}v^r-f^r_{i+1,i+1}v^r&\text{if }1\leq i\leq n-1,\\
f^r_{n-1,n-1}v^r+f^r_{n,n}v^r&\text{if }i=n\text{ and $r$ is even},\\
-\bigl(f^r_{1,1}v^r+f^r_{2,2}v^r\bigr)+2u^{-1}v^{r}du&\text{if }i=0\text{ and $r$ is even},
\end{cases}\\
X^+_{i,r}=\begin{cases}
f^r_{i,i+1}v^r&\text{if }1\leq i\leq n-1,\\
f^r_{n-1,-n}v^r&\text{if }i=n,\\
(-1)^rf^r_{-2,1}uv^r&\text{if }i=0,\\
f^r_{n,-n}v^r&\text{if }i=n+1\text{ and $r$ is odd},\\
f^r_{-1,1}uv^r&\text{if }i=-1\text{ and $r$ is odd},
\end{cases}\\
X^-_{i,r}=\begin{cases}
f^r_{i+1,i}v^r&\text{if }1\leq i\leq n-1,\\
f^r_{-n,n-1}v^r&\text{if }i=n,\\
(-1)^rf^r_{1,-2}u^{-1}v^r&\text{if }i=0,\\
f^r_{-n,n}v^r&\text{if }i=n+1\text{ and $r$ is odd},\\
f^r_{1,-1}u^{-1}v^r&\text{if }i=-1\text{ and $r$ is odd}.
\end{cases}
\end{gather*}
\begin{Theorem}\label{inj}
The homomorphism $\phi$ is injective.
\end{Theorem}
This section is devoted to the proof of Theorem~\ref{inj}.
By Theorem~\ref{Thm128}, we can define $\smash{\widetilde{f}_{i,j}^{r,s}\in \mathfrak{D}(2n)}$ for $i\neq j$ in the same way as $f_{i,j}^{r,s}\in\mathfrak{sl}(2n)\bigl[u^{\pm1},v\bigr]^\tau$. By Theorem~\ref{Thm128}, we obtain the relations
\begin{gather*}
\bigl[H_{i,x},\widetilde{f}_{p,q}^{r,s}\bigr]\nonumber\\
=\begin{cases}
(\delta_{i,p}-\delta_{i+1,p}-\delta_{i,q}+\delta_{i+1,q}-(-1)^x(\delta_{-i,p}-\delta_{-i,q}-\delta_{-i-1,p}+\delta_{-i-1,q}))\widetilde{f}_{p,q}^{r,s+x}\\
\quad \text{if }1\leq i\leq n-1,\\
(\delta_{n-1,p}+\delta_{n,p}-\delta_{n-1,q}-\delta_{n,q}-(-1)^x(\delta_{-(n-1),p}-\delta_{-(n-1),q}+\delta_{-n,p}-\delta_{-n,q}))\widetilde{f}_{p,q}^{r,s+x}\\
\quad \text{if }i=n,\\
-(\delta_{1,p}+\delta_{2,p}-\delta_{1,q}-\delta_{2,q}-(-1)^x(\delta_{-1,p}-\delta_{-1,q}+\delta_{-2,p}-\delta_{-2,q}))\widetilde{f}_{p,q}^{r,s+x}\\
\quad \text{if }i=0,
\end{cases}\\
\bigl[\widetilde{f}_{i,j}^{r_1,s_1},\widetilde{f}_{p,q}^{r_2,s_2}\bigr]
=\delta_{j,p}\widetilde{f}_{i,q}^{r_1+r_2,s_1+s_2}-\delta_{i,q}\widetilde{f}_{p,j}^{r_1+r_2,s_1+s_2}\nonumber\\
\quad{}
-(-1)^{s_1}\delta_{i,-p}\widetilde{f}_{-j,q}^{r_1+r_2,s_1+s_2}+(-1)^{s_1}\delta_{j,-q}\widetilde{f}_{p,-i}^{r_1+r_2,s_1+s_2}\qquad \text{if }(i,j)\neq(q,p),(-p,-q).
\end{gather*}

By Theorem~\ref{Thm128}, we find that the kernel of $\Phi$ is contained in $\bigoplus_{r\neq0,k\in\mathbb{Z}}\mathfrak{D}(2n)_{r\delta}^k$.
\begin{Lemma}
For $r\geq0$, the subspace $\mathfrak{D}(2n)_{r\delta}^k$ is spanned by
\begin{align*}
a(r,k)={}&\bigl\{\bigl[X^+_{i,0},\widetilde{f}_{i+1,i}^{r,k}\bigr]\mid1\leq i\leq n-1\bigr\}\cup\bigl\{\bigl[X^+_{n,0},\widetilde{f}_{-n,n-1}^{r,k}\bigr]\mid1\leq i\leq n-1\bigr\}\\
&{\cup}\,\bigl\{\bigl[X^+_{0,s},\widetilde{f}_{1,-2}^{r,k}\bigr]\mid s=0,1\bigr\}\cup\bigl\{\bigl[X^+_{1,1},\widetilde{f}_{2,1}^{r,k-1}\bigr]-\bigl[X^+_{1,0},\widetilde{f}_{2,1}^{r,k}\bigr]\bigr\}.
\end{align*}
\end{Lemma}
\begin{proof}
By Theorem~\ref{Thm128}, we find that $\mathfrak{D}(2n)_{r\delta}^k$ is spanned by the set
\begin{align*}
&\bigl\{\bigl[X^+_{i,s},\widetilde{f}_{i+1,i}^{r,v}\bigr]\mid1\leq i\leq n-1,s+v=k\bigr\}\cup\bigl\{\bigl[X^+_{n,s},\widetilde{f}_{-n,n-1}^{r,v}\bigr]\mid s+v=k\bigr\}\\
&\qquad{}\cup\bigl\{\bigl[X^+_{0,s},\widetilde{f}_{1,-2}^{r-1,v}\bigr]\mid s+v=k\bigr\}.
\end{align*}
Since $\mathfrak{D}(2n)$ is a central extension of $\mathfrak{sl}(2n)\bigl[u^{\pm1},v\bigr]^\tau$, we find that
\begin{gather*}
\bigl[X^+_{i,s},\widetilde{f}_{i+1,i}^{r,v}\bigr]-\bigl[X^+_{i,0},\widetilde{f}_{i+1,i}^{r,s+v}\bigr],\\
\bigl[X^+_{n,s},\widetilde{f}_{-n,n-1}^{r,v}\bigr]-\bigl[X^+_{n,0},\widetilde{f}_{-n,n-1}^{r,s+v}\bigr],\\
\bigl[X^+_{0,s},\widetilde{f}_{1,-2}^{r-1,v}\bigr]-\bigl[X^+_{0,0},\widetilde{f}_{1,-2}^{r-1,s+v}\bigr]
\end{gather*}
are central elements of $\mathfrak{D}(2n)$. Thus, we obtain
\begin{align}
0&=\bigl[H_{i,1},\bigl[X^+_{i,s},\widetilde{f}_{i+1,i}^{r,v}\bigr]\bigr]-\bigl[H_{i,1},\bigl[X^+_{i,0}, \widetilde{f}_{i+1,i}^{r,s+v}\bigr]\bigr]=2\bigl[X^+_{i,s+1},\widetilde{f}_{i+1,i}^{r,v}\bigr] -2\bigl[X^+_{i,s},\widetilde{f}_{i+1,i}^{r,v+1}\bigr]\nonumber\\
&\quad-2\bigl[X^+_{i,1},\widetilde{f}_{i+1,i}^{r,s+v}\bigr]+2\bigl[X^+_{i,0},\widetilde{f}_{i+1,i}^{r,s+v+1}\bigr]\qquad \text{for }1\leq i\leq n-1,\label{eq92}\\
0&=\bigl[H_{n-1,1},\bigl[X^+_{n,s},\widetilde{f}_{-n,n-1}^{r,v}\bigr]\bigr]-\bigl[H_{n-1,1}, \bigl[X^+_{n,0},\widetilde{f}_{-n,n-1}^{r,s+v}\bigr]\bigr]\nonumber\\
&=2\bigl[X^+_{n,s+1},\widetilde{f}_{-n,n-1}^{r,v}\bigr]-2\bigl[X^+_{n,s},\widetilde{f}_{-n,n-1}^{r,v+1}\bigr] -2\bigl[X^+_{n,1},\widetilde{f}_{-n,n-1}^{r,s+v}\bigr]+2\bigl[X^+_{n,0},\widetilde{f}_{-n,n-1}^{r,s+v+1}\bigr],\nonumber\\
0&=\bigl[H_{1,1},\bigl[X^+_{0,s},\widetilde{f}_{1,-2}^{r-1,v}\bigr]\bigr]-\bigl[H_{1,1},\bigl[X^+_{0,0},\widetilde{f}_{1,-2}^{r-1,s+v}\bigr]\bigr]\nonumber\\
&=2\bigl[X^+_{0,s+1},\widetilde{f}_{1,-2}^{r-1,v}\bigr]-2\bigl[X^+_{0,s},\widetilde{f}_{1,-2}^{r-1,v+1}\bigr] -2\bigl[X^+_{0,1},\widetilde{f}_{1,-2}^{r-1,s+v}\bigr]+2\bigl[X^+_{0,0},\widetilde{f}_{1,-2}^{r-1,s+v+1}\bigr].\nonumber
\end{align}
Thus, $\mathfrak{A}(2n)_{r\delta}^k$ is spanned by
\begin{align*}
&\{\bigl[X^+_{i,s},\widetilde{f}_{i+1,i}^{r,k-s}\bigr]\mid1\leq i\leq n-1,\,s=0,1\}\\
& \qquad\cup\{\bigl[X^+_{n,s},\widetilde{f}_{-n,n-1}^{r,k-s}\bigr]\mid1\leq i\leq n-1,s=0,1\}\cup\{\bigl[X^+_{0,s},\widetilde{f}_{1,-2}^{r-1,k-s}\bigr]\mid s=0,1\}.
\end{align*}
By Corollary~\ref{Cor25} and a direct computation, we find that, by composing $\tau_i$, we can construct~$\widetilde{\tau}_1$,~$\widetilde{\tau}_2$ satisfying that
\begin{alignat*}{3}
& \widetilde{\tau}_1\bigl(X^+_{i,k}\bigr)=X^+_{1,k},\qquad && \widetilde{\tau}_1\bigl(\widetilde{f}_{i,i+1}^{r,k}\bigr)=\widetilde{f}_{2,1}^{r,k},&\\
&\widetilde{\tau}_2\bigl(X^+_{0,k}\bigr)=X^+_{1,k},\qquad && \widetilde{\tau}_2\bigl(\widetilde{f}_{1,-2}^{r-1,k}\bigr)=\widetilde{f}_{2,1}^{r,k},&\\
&\widetilde{\tau}_3\bigl(X^+_{n,k}\bigr)=X^+_{1,k},\qquad && \widetilde{\tau}_3\bigl(\widetilde{f}_{-n,n-1}^{r,k}\bigr)=\widetilde{f}_{2,1}^{r,k}&
\end{alignat*}
for $1\leq i\leq n$. Then, we obtain
\begin{gather}
\bigl[X^+_{1,1},\widetilde{f}_{2,1}^{r,k-1}\bigr]-\bigl[X^+_{1,0},\widetilde{f}_{2,1}^{r,k}\bigr]
=\bigl[\widetilde{\tau}_1\bigl(X^+_{i,1}\bigr),\widetilde{\tau}_1\bigl(\widetilde{f}_{i+1,i}^{r,k-1}\bigr)\bigr]-\bigl[\widetilde{\tau}_1\bigl(X^+_{i,0}\bigr),\widetilde{\tau}_1\bigl(\widetilde{f}_{i+1,i}^{r,k}\bigr)\bigr]\nonumber\\
\hphantom{\bigl[X^+_{1,1},\widetilde{f}_{2,1}^{r,k-1}\bigr]-\bigl[X^+_{1,0},\widetilde{f}_{2,1}^{r,k}\bigr]}{}
=\bigl[X^+_{i,1},\widetilde{f}_{i+1,i}^{r,k-1}\bigr]-\bigl[X^+_{i,0},\widetilde{f}_{i+1,i}^{r,k}\bigr],\label{eq94}\\
\bigl[X^+_{1,1},\widetilde{f}_{2,1}^{r,k-1}\bigr]-\bigl[X^+_{1,0},\widetilde{f}_{2,1}^{r,k}\bigr]
=\bigl[\widetilde{\tau}_2\bigl(X^+_{0,1}\bigr),\widetilde{\tau}_2\bigl(\widetilde{f}_{1,-2}^{r-1,k-1}\bigr)\bigr]-\bigl[\widetilde{\tau}_2\bigl(X^+_{0,0}\bigr),\widetilde{\tau}_2\bigl(\widetilde{f}_{1,-2}^{r-1,k}\bigr)\bigr]\nonumber\\
\hphantom{\bigl[X^+_{1,1},\widetilde{f}_{2,1}^{r,k-1}\bigr]-\bigl[X^+_{1,0},\widetilde{f}_{2,1}^{r,k}\bigr]}{}
=\bigl[X^+_{0,1},\widetilde{f}_{1,-2}^{r-1,k-1}\bigr]-\bigl[X^+_{0,0},\widetilde{f}_{1,-2}^{r-1,k}\bigr],\nonumber\\
\bigl[X^+_{1,1},\widetilde{f}_{2,1}^{r,k-1}\bigr]-\bigl[X^+_{1,0},\widetilde{f}_{2,1}^{r,k}\bigr]
=\bigl[\widetilde{\tau}_3\bigl(X^+_{n,1}\bigr),\widetilde{\tau}_3\bigl(\widetilde{f}_{-n,n-1}^{r,k-1}\bigr)\bigr]
-\bigl[\widetilde{\tau}_3\bigl(X^+_{n,0}\bigr),\widetilde{\tau}_3\bigl(\widetilde{f}_{-n,n-1}^{r,k}\bigr)\bigr]\nonumber\\
\hphantom{\bigl[X^+_{1,1},\widetilde{f}_{2,1}^{r,k-1}\bigr)\bigr]-\bigl[X^+_{1,0},\widetilde{f}_{2,1}^{r,k}\bigr)\bigr]}{}
=\bigl[X^+_{n,1},\widetilde{f}_{-n,n-1}^{r,k-1}\bigr]-\bigl[X^+_{n,0},\widetilde{f}_{-n,n-1}^{r,k}\bigr],\nonumber
\end{gather}
because the right-hand sides are central elements of $D(2n)$.
\end{proof}

Now, we begin the proof of Theorem~\ref{inj}. We only consider the case that $r>0$. When $k$ is even, the images of $a(r,k)$ via $\Phi$ are linearly independent. However, in the case that $k$ is odd, we obtain
\begin{equation*}
\bigl[\Phi\bigl(X^+_{1,1}\bigr),\Phi\bigl(\widetilde{f}_{2,1}^{r,k}\bigr)\bigr]-\bigl[\Phi\bigl(X^+_{1,0}\bigr),\Phi\bigl(\widetilde{f}_{2,1}^{r,k}\bigr)\bigr]=0.
\end{equation*}
Thus, we need to show \smash{$\bigl[X^+_{1,1},\widetilde{f}_{2,1}^{r,2s}\bigr]-\bigl[X^+_{1,0},\widetilde{f}_{2,1}^{r,2s+1}\bigr]=0$}.
The case $s=0$ is precisely the defining relation.
Next, we consider the case $s\geq1$. By a direct computation, we obtain
\begin{align*}
\bigl[X^+_{n,2},f_{-n,n-1}^{u,2s-1}\bigr]&=-\frac{1}{2}\bigl[X^+_{n,2},\bigl[X^+_{n-1,0},f_{-(n-1),n-1}^{u,2s-1}\bigr]\bigr]\\
&=-\frac{1}{2}\bigl[X^+_{n-1,0},\bigl[X^+_{n,2},f_{-(n-1),n-1}^{u,2s-1}\bigr]\bigr]=\bigl[X^+_{n-1,0},f_{n,n-1}^{u,2s+1}\bigr]
\end{align*}
and
\begin{align*}
\bigl[X^+_{n,0},f_{-n,n-1}^{u,2s+1}\bigr]&=-\frac{1}{2}\bigl[X^+_{n,0},\bigl[X^+_{n-1,2},f_{-(n-1),n-1}^{u,2s-1}\bigr]\bigr]\\
&=-\frac{1}{2}\bigl[X^+_{n-1,2},\bigl[X^+_{n,0},f_{-(n-1),n-1}^{u,2s-1}\bigr]\bigr]=\bigl[X^+_{n-1,2},f_{n,n-1}^{u,2r-1}\bigr].
\end{align*}
Thus, we obtain
\begin{equation}
\bigl[X^+_{n,2},f_{-n,n-1}^{u,2s-1}\bigr]-\bigl[X^+_{n,0},f_{-n,n-1}^{u,2s+1}\bigr]=\bigl[X^+_{n-1,0},f_{n,n-1}^{u,2s+1}\bigr]-\bigl[X^+_{n-1,2},f_{n,n-1}^{u,2s-1}\bigr].\label{eq96}
\end{equation}
By \eqref{eq92}, \eqref{eq96} and \eqref{eq94}, we find the relations
\begin{gather*}
\bigl[X^+_{n,2},f_{-n,n-1}^{u,2s-1}\bigr]-\bigl[X^+_{n,0},f_{-n,n-1}^{u,2s+1}\bigr]=0,\\
\bigl[X^+_{1,1},\widetilde{f}_{2,1}^{2r,2s-1}\bigr]-\bigl[X^+_{1,0},\widetilde{f}_{2,1}^{2r+1,k}\bigr]=0.
\end{gather*}
We complete the proof of Theorem~\ref{inj}.

By the same proof as in~\cite[Section 4]{HU}, the associative algebra $D(2n)$ has the following finite presentation.
\begin{Definition}
We define \smash{$\widetilde{D}(2n)$} as the associative algebra generated by
\begin{equation*}
\bigl\{h_{i,r},x^\pm_{i,r},h_{n,0},\mid0\leq i\leq n-1,\,1\leq j\leq n,\,r=0,1\bigr\}
\end{equation*}
with the relations
\begin{gather*}
[h_{i,r},h_{j,s}]=0,\\
\bigl[h_{i,0},x^\pm_{j,0}\bigr]=\pm a^r_{i,j}x^\pm_{j,0},\\
\bigl[x^+_{i,0},x^-_{j,0}\bigr]=\delta_{i,j}h_{i,0},\\
\bigl[x^+_{i,1},x^-_{i,0}\bigr]=\begin{cases}
h_{i,1}&\text{if }1\leq i\leq n-1,\\
h_{n-1,1}&\text{if }i=1,\\
-h_{1,r+s}&\text{if }j=0,
\end{cases}\\
\bigl[x^\pm_{i,1},x^\pm_{i,0}\bigr]=0,\\
\bigl[x^\pm_{i,0},x^\pm_{j,0}\bigr]=0\qquad \text{if }a^0_{i,j}=0,\\
\bigl[x^\pm_{i,0},\bigl[x^\pm_{i,0},x^\pm_{j,0}\bigr]\bigr]=0\qquad \text{if }a^0_{i,j}=-1,\\
\bigl[\bigl[x^+_{n-1,1},x^-_{n-1,1}\bigr],x^\pm_{n,0}\bigr]=0,\\
\bigl[\bigl[x^+_{0,1},x^-_{0,1}\bigr],x^\pm_{1,0}\bigr]=0,\\
\bigl[x^+_{1,1},\bigl[x^-_{1,0},(E_{1,1}-E_{2,2})t^u\bigr]\bigr]-\bigl[x^+_{1,0},\bigl[x^-_{1,1},(E_{1,1}-E_{2,2})t^u\bigr]\bigr]=0.
\end{gather*}
\end{Definition}
\begin{Theorem}
The associative algebra $\widetilde{D}(2n)$ is isomorphic to $D(2n)$.
\end{Theorem}

\subsection*{Acknowledgements}
The author expresses his sincere thanks to Shuichi Harako for helping him to compute OPEs. He is also grateful to the anonymous referees for their careful reading of the manuscript and for their valuable comments and suggestions, which have helped improve the paper. This research was supported by JSPS Research Fellowship for Young Scientists (PD), Grant Number JP25KJ0038.

\pdfbookmark[1]{References}{ref}
 \LastPageEnding


\begin{thebibliography}{99}
\footnotesize\itemsep=0pt

\bibitem{BR1}
Belliard S., Regelskis V., Drinfeld~{J} presentation of twisted {Y}angians,
 \href{https://doi.org/10.3842/SIGMA.2017.011}{\textit{SIGMA}} \textbf{13}
 (2017), 011, 35~pages,
 \href{https://arxiv.org/abs/1401.2143}{arXiv:1401.2143}.

\bibitem{Bro}
Brown J., Twisted {Y}angians and finite {$W$}-algebras,
 \href{https://doi.org/10.1007/s00031-008-9041-x}{\textit{Transform. Groups}}
 \textbf{14} (2009), 87--114,
 \href{https://arxiv.org/abs/0710.2918}{arXiv:0710.2918}.

\bibitem{BK}
Brundan J., Kleshchev A., Shifted {Y}angians and finite {$W$}-algebras,
 \href{https://doi.org/10.1016/j.aim.2004.11.004}{\textit{Adv. Math.}}
 \textbf{200} (2006), 136--195,
 \href{https://arxiv.org/abs/math/0407012}{arXiv:math/0407012}.

\bibitem{DKV}
De~Sole A., Kac V.G., Valeri D., A~{L}ax type operator for quantum finite
 {$W$}-algebras,
 \href{https://doi.org/10.1007/s00029-018-0439-6}{\textit{Selecta
 Math.~(N.S.)}} \textbf{24} (2018), 4617--4657,
 \href{https://arxiv.org/abs/1707.03669}{arXiv:1707.03669}.

\bibitem{D1}
Drinfeld V.G., Hopf algebras and the quantum {Y}ang--{B}axter equation,
 \textit{Soviet Math. Dokl.} \textbf{32} (1985), 256--258.

\bibitem{D2}
Drinfeld V.G., A~new realization of {Y}angians and of quantum affine algebras,
 \textit{Soviet Math. Dokl.} \textbf{36} (1988), 212--216.

\bibitem{Gu1}
Guay N., Affine {Y}angians and deformed double current algebras in type~{$A$},
 \href{https://doi.org/10.1016/j.aim.2006.08.007}{\textit{Adv. Math.}}
 \textbf{211} (2007), 436--484.

\bibitem{GNW}
Guay N., Nakajima H., Wendlandt C., Coproduct for {Y}angians of affine
 {K}ac--{M}oody algebras,
 \href{https://doi.org/10.1016/j.aim.2018.09.013}{\textit{Adv. Math.}}
 \textbf{338} (2018), 865--911,
 \href{https://arxiv.org/abs/1701.05288}{arXiv:1701.05288}.

\bibitem{HU}
Harako S., Ueda M., A~new presentation of the twisted {Y}angian of type~{$D$},
 \href{https://arxiv.org/abs/2507.00350}{arXiv:2507.00350}.

\bibitem{K}
Kac V.G., Infinite-dimensional {L}ie algebras, 3rd ed.,
 \href{https://doi.org/10.1017/CBO9780511626234}{Cambridge University Press},
 Cambridge, 1990.

\bibitem{Kassel}
Kassel C., K\"ahler differentials and coverings of complex simple {L}ie
 algebras extended over a commutative algebra,
 \href{https://doi.org/10.1016/0022-4049(84)90040-9}{\textit{J.~Pure Appl.
 Algebra}} \textbf{34} (1984), 265--275.

\bibitem{L}
Lu K., Minimalistic presentation and coideal structure of twisted {Y}angians,
 \href{https://doi.org/10.1007/s00220-026-05615-3}{\textit{Comm. Math. Phys.}}
 \textbf{407} (2026), 99, 38~pages,
 \href{https://arxiv.org/abs/2511.07136}{arXiv:2511.07136}.

\bibitem{LPTTW}
Lu K., Peng Y.N., Tappeiner L., Topley L., Wang W., Shifted twisted {Y}angians
 and finite {$W$}-algebras of classical type,
 \href{https://arxiv.org/abs/2505.03316}{arXiv:2505.03316}.

\bibitem{LWZ0}
Lu K., Wang W., Zhang W., A~{D}rinfeld type presentation of twisted {Y}angians,
 \href{https://doi.org/10.1090/ert/702}{\textit{Represent. Theory}}
 \textbf{29} (2025), 838--870,
 \href{https://arxiv.org/abs/2308.12254}{arXiv:2308.12254}.

\bibitem{LWZ}
Lu K., Wang W., Zhang W., Affine {$\imath$}quantum groups and twisted
 {Y}angians in {D}rinfeld presentations,
 \href{https://doi.org/10.1007/s00220-025-05263-z}{\textit{Comm. Math. Phys.}}
 \textbf{406} (2025), 98, 36~pages,
 \href{https://arxiv.org/abs/2406.05067}{arXiv:2406.05067}.

\bibitem{LWZ1}
Lu K., Zhang W., A~{D}rinfeld type presentation of twisted {Y}angians of
 quasi-split type,
 \href{https://doi.org/10.1142/S021919972650001X}{\textit{Commun. Contemp.
 Math.}}, {t}o appear,
 \href{https://arxiv.org/abs/2408.06981}{arXiv:2408.06981}.

\bibitem{MNT}
Matsuo A., Nagatomo K., Tsuchiya A., Quasi-finite algebras graded by
 {H}amiltonian and vertex operator algebras, in Moonshine: the First Quarter
 Century and Beyond, \textit{London Math. Soc. Lecture Note Ser.}, Vol.~372,
 \href{https://doi.org/10.1017/CBO9780511730054.015}{Cambridge University
 Press}, Cambridge, 2010, 282--329,
 \href{https://arxiv.org/abs/math/0505071}{arXiv:math/0505071}.

\bibitem{MRY}
Moody R.V., Rao S.E., Yokonuma T., Toroidal {L}ie algebras and vertex
 representations, \href{https://doi.org/10.1007/BF00147350}{\textit{Geom.
 Dedicata}} \textbf{35} (1990), 283--307.

\bibitem{O}
Olshanskii G.I., Twisted {Y}angians and infinite-dimensional classical {L}ie
 algebras, in Quantum Groups ({L}eningrad, 1990), \textit{Lecture Notes in
 Math.}, Vol.~1510, \href{https://doi.org/10.1007/BFb0101183}{Springer},
 Berlin, 1992, 104--119.

\bibitem{U77}
Ueda M., Twisted affine {Y}angian and rectangular {$W$}-algebra of type~{$D$},
 \href{https://doi.org/10.1007/s10468-025-10348-3}{\textit{Algebr. Represent.
 Theory}} \textbf{28} (2025), 1195--1224,
 \href{https://arxiv.org/abs/2107.09999}{arXiv:2107.09999}.

\bibitem{U78}
Ueda M., Suggestion of a~definition of the twisted affine Yangian of
 type~{$D$}, \href{https://arxiv.org/abs/2512.06340}{arXiv:2512.06340}.

\end{thebibliography}
\end{document}